\documentclass{siamart171218}
\usepackage{float}
\usepackage{amsmath}
\usepackage{amsfonts}
\usepackage{paralist}
\usepackage{amsfonts}
\usepackage{enumerate}
\usepackage{amssymb}
\usepackage{array}
\usepackage{listings}
\usepackage{color}
\usepackage{mathtools}
\usepackage{setspace}
\usepackage{algorithmicx}
\usepackage{varwidth}
\usepackage{blkarray}
\usepackage{enumitem}
\usepackage{times} 
\usepackage{amscd,amsmath,amsfonts}
\usepackage{setspace}
\usepackage{subfigure}
\usepackage{graphics}
\usepackage{graphicx,epsfig}
\usepackage{psfrag}
\usepackage{color}
\usepackage{tabularx}
\usepackage{array}
\usepackage[left=3.0cm,top=3.1cm,right=3.0cm,bottom=3.1cm,bindingoffset=0.5cm]{geometry}

\newtheorem{cor}{Corollary}
\newtheorem{prob}{Problem}

\newcommand{\A}{\mathcal{A}}
\newcommand{\B}{\mathcal{B}}
\newcommand{\C}{\mathcal{C}}
\newcommand{\K}{\mathcal{K}_e}
\newcommand{\W}{\mathcal{W}}
\newcommand{\M}{\mathcal{M}}

\providecommand{\rset}[1]{\mathbb{R}^}
\newcommand{\fin}{\hfill$\bigtriangleup$}

\DeclareMathOperator{\diag}{diag}

\DeclareMathOperator{\opd}{d\!}

\usepackage[colorinlistoftodos,bordercolor=orange,backgroundcolor=orange!20,linecolor=orange,textsize=scriptsize]{todonotes}

\headers{Optimal $H_2$ moment matching-based model reduction}{I. Necoara and T.~C. Ionescu}

\title{Optimal $H_2$ moment matching-based model reduction for linear systems by (non)convex optimization\thanks{Submitted to the editors DATE.
\funding{This work  is supported by the Executive Agency for Higher Education, Research and Innovation Funding (UEFISCDI), Romania,  PNIII-P4-PCE-2016-0731, project ScaleFreeNet, no. 39/2017, and partially supported by University Politehnica Bucharest through the internal research grant GNaC 2918 ARUT,  no. 4/15.10.2018.}}}

\author{Ion Necoara\thanks{Department of Automatic Control and Systems Engineering, University Politehnica Bucharest (UPB), 060042 Bucharest, Romania   (\email{ion.necoara@acse.pub.ro}).}
\and Tudor C. Ionescu\thanks{Department of Automatic Control and Systems Engineering, University Politehnica Bucharest (UPB), 060042 Bucharest and  "Gh. Mihoc-C. Iacob" Institute of Mathematical Statistics and Applied Mathematics of the Romanian Academy, 050711 Bucharest, Romania (\email{tudor.ionescu@acse.pub.ro}).}}

\begin{document}
\maketitle

\begin{abstract}
In this paper we compute families of reduced order models that match a prescribed set of $\nu$ moments of a highly dimensional linear time-invariant system.  First, we fully parametrize the models in the interpolation points \emph{and} in the free parameters, and then we fix the set of interpolation points and  parametrize the models only in the free parameters. Based on these two parametrizations and using as objective function  the $H_2$-norm of  the error approximation we derive  non-convex optimization problems, i.e., we search for the optimal free parameters and even the interpolation points   to determine the approximation model yielding the minimal $H_2$-norm error.  Further, we  provide the necessary first-order optimality conditions for these optimization problems given explicitly  in terms of the controllability and the observability Gramians of a minimal realization of the error system. Using the optimality conditions, we propose  gradient type methods for solving the corresponding optimization problems, with mathematical guarantees on their  convergence. We also derive convex SDP relaxations for these problems and analyze when the convex relaxations are exact.  We illustrate numerically the efficiency of our results on several test examples.
\end{abstract}


\noindent

\section{Introduction}
Today we are living in a complex and interconnected world. Mathematical tools yield complex and highly dimensional dynamical models, e.g., from partial-differential equations or networks  of interconnected subsystems.  Hence, for purposes such as simulation and control design, scientists and engineers need tweaking of such models rendering them simpler and useful. To this end, model reduction is called for. The main idea of model order reduction is to find a low-order mathematical model that approximates the given highly dimensional  dynamical system. The approximation is accurate if the \emph{approximation error} is \emph{small} and if the most important physical properties/structure, such as the stability of the given system are preserved. A large number of methods have been developed for linear systems, split in two major categories. The first category consists of the so-called SVD-based methods, such as balanced truncation and Hankel norm approximation, described, e.g., in \cite{moore-1981}. The second category contains moment matching-based methods as, e.g., in  \cite{antoulas-sorensen-1999, grimme-1997, vandooren-1995}.  For a  survey on model reduction of linear systems see the monograph  \cite{antoulas-2005}.  

\medskip

\noindent \textit{State-of-the-art}.  Balancing is a tool using an energy measure of the states of the system to determine whether that state can be neglected in the dynamics or not, introduced by Moore in \cite{moore-1981} for stable linear systems. It measures the controllability and the observability of a given state through the Hankel singular values, input-output invariants of the system. From a model reduction point of view, one may choose to truncate the states which are badly controllable and badly observable, corresponding to the smaller Hankel singular values, yielding a lower-order dynamical system. The balanced truncation model preserves the stability of the given system. Another important property of this approach is the analitic upper error bound found of the infinity norm of the error-system, see, e.g., the work of Glover \cite{glover-1984}. However, the balanced truncation-based approximation does not minimize any norm \cite{antoulas-2005} associated to the error system. A refinement of the balanced truncation leads to an approximation method, with respect to the 2-induced norm of the Hankel operator associated to the system, known as optimal approximation in the Hankel norm. 

\medskip

\noindent The second category of model reduction techniques is based on  moment matching. Model reduction moment matching techniques represent an efficient tool for reducing the dimension of the system, see, e.g., \cite{devillemagne-skelton-IJC1987,antoulas-ball-kang-willems-LINALG1990} and \cite{antoulas-2005} for an overview for linear systems. In such techniques the reduced order model is obtained by constructing a lower degree rational function that approximates the given higher degree transfer function. The low degree rational function matches the original transfer function and its derivatives at various points in the complex plane. The notion of moments has been given in \cite{antoulas-2005}, through the series expansion of the transfer function of the linear system, see also \cite{gragg-lindquist-LAA1983,vandooren-1995,gallivan-vandooren-NA1999}. Hence, one can write equivalent definitions of moments. A first equivalent definition is in terms of  \emph{right Krylov projections} and a second equivalent definition is the dual of the previous one, in terms of \emph{left Krylov projections}. The reduced order models obtained through Krylov projections match a prescribed number of moments, say $\nu$. Alternatively, in \cite{gallivan-vandendorpe-vandooren-JCAM2004}, the Krylov projections are obtained solving Sylvester equations. To improve the accuracy of the reduced order models that achieve moment matching, two-sided projections have been employed, see, e.g., \cite{antoulas-2005,devillemagne-skelton-IJC1987,grimme-sorensen-vandooren-NA1995}. The simultaneous application of the left and the right projections yields a  reduced order model that matches $2 \nu$ moments at two sets of $\nu$ interpolation points, respectively. Recently, in \cite{gugercin-antoulas-beattie-SIAM2008,anic-beattie-gugercin-antoulas-AUT2013}, using two-sided projection-based interpolatory methods, the model that minimizes the $H_2$-norm of the approximation error is computed. Here, a unifying framework for the optimal $H_2$ approximant has been obtained using best approximation properties in the underlying Hilbert space. A set of local optimality conditions taking the form of a structured orthogonality condition have been developed. Based on the interpolation framework, \cite{gugercin-antoulas-beattie-SIAM2008} has provided an iteratively corrected rational Krylov algorithm (IRKA) for $H_2$ model reduction. The resulting model interpolates the transfer function of the given system and its derivatives at the mirror images of the poles of the approximant, see also \cite{meier-luenberger-TAC1967}. Further, in \cite{mayo-antoulas-LAA2007}, a new framework has been proposed for the solution to the realization problem. More precisely, the moment matching problem has been recast in terms of the Loewner matrix and the solutions of Sylvester equations, with matrices constructed from tangential interpolation data. The result is a reduced order model that achieves moment matching and is minimal, while the corresponding  Loewner/Krylov matrix-based algorithm is highly efficient numerically, involving only matrix-vector multiplications.  

\medskip

\noindent \textit{Related work}.  Most recently, in \cite{astolfi-CDC2007,astolfi-TAC2010,i-astolfi-colaneri-SCL2014} new definitions of moments in a time-domain framework have been given. Algebraically, the moments of a linear system are defined in terms of the solution of a Sylvester equation. From a systemic viewpoint, moments are in (a one-to-one) relation with the well-defined steady-state response of the system driven by a signal generator (a novel interpretation of the results in \cite{gallivan-vandendorpe-vandooren-JCAM2004}). An approximation achieves moment matching if the steady-state of its response to a signal generator matches the steady-state response of the original system at the same signal generator. Imposing such moment matching conditions yields a \emph{family of} reduced order \emph{parametrized} approximations. The degrees of freedom are used such that  properties, like stability, are preserved. Based on a dual Sylvester equation, a new definition of moments dual to the previous one results, related to the well-defined steady-state response of a generalized signal generator driven by the system. The resulting (family of) reduced order models that achieve moment matching are also  parametric. Employing both time-domain notions of moments, two-sided moment matching can be also achieved \cite{i-TAC2016}. The resulting model that matches $2\nu$ moments is computed by a specific selection of the free parameters. Moreover, the reduced order model matching the moments of both the transfer function of the system and its derivative is determined by another specific choice of the free parameters.  The results therein follow the necessary  first-order  optimality conditions of optimal $H_2$ model reduction. Experimentally, these models exhibit low $H_2/H_\infty$-norm of the error following the arguments in \cite{gugercin-antoulas-beattie-SIAM2008} that a reduced order  locally minimizing the $H_{2}$-norm of the approximation error  is achieving moment matching of both the transfer function of the given system and its derivative. 

\medskip

\noindent \textit{Motivation}.  However, in the time-domain moment matching framework finding a model in the family of $\nu$ order models that matches a set of $\nu$ prescribed moments and approximates a system with the \emph{minimal} $H_2$-norm of the error is an open question.   Even the relaxed version of this general  optimal $H_2$ model reduction problem, where we seek only  the free parameters that yield a model from \emph{a family} of $\nu$ order models that match $\nu$ prescribed moments at a set of $\nu$ \emph{fixed} interpolation points and  minimizing  the $H_2$-norm of the approximation error  has not been  addressed yet. Some initial  progress has been made in \cite{i-TAC2016}, see related work paragraph, by matching the derivatives of the given system. However, the reduced model does not yield the optimal $H_2$ approximation since there is no degree of freedom left in the parameters to ensure a minimal norm of the error. Hence, in the time-domain moment matching framework, the problem of finding the free parameters and eventually also the interpolation points yielding a $\nu$ order model that matches $\nu$ moments and minimizes the $H_2$-norm of the approximation error has not been addressed or fully understood.  Furthermore,  the algorithms developed so far for the problem of optimal $H_2$-norm model reduction do not guarantee preservation of stability or other physical properties. Thus, the problem of finding the \emph{stable} reduced order model achieving the minimal possible approximation  for $H_2$-norm motivates our work here.

\medskip

\noindent  \textit{Contributions}.   In this paper we write families of reduced order models that match a prescribed set of $\nu$ moments of a highly dimensional  linear time-invariant system.  First, we fully parametrize the models in the interpolation points \emph{and} in the free parameters, and then \emph{we fix} the set of interpolation points and  parametrize the models only in the free parameters. Based on the parametrizations and using as objective function  the $H_2$-norm of  the error approximation we derive non-convex optimization problems, i.e., we search for the optimal free parameters and eventually  the interpolation points to determine the approximation model yielding the minimal $H_2$-norm error.  For all the optimization problems we compute the necessary first-order optimality conditions given explicitly  in terms of the controllability and the observability Gramians of a minimal realization of the error system. Furthermore, using the optimality conditions,  we propose several  gradient-type methods for solving the corresponding optimization problems, with mathematical guarantees on their  convergence. We also provide convex SDP relaxations for these non-convex  optimization problems and analyze when these  relaxations are exact. Our contributions are summarized as follows: 

\begin{enumerate}[wide]
\item[\textbf{(i)}] We first formulate a general model reduction problem with reduced models  from the family of models matching $\nu$ moments parameterized  in the interpolation points and in the free parameters.  A corresponding optimization formulation is derived, where the objective function is the  $H_2$-norm of the approximation error, written explicitly  in terms of the controllability and observability Gramians of a minimal realization of the error system.  We also  write the necessary first-order optimality conditions (KKT system) of this optimization problem, in terms of these Gramians. 

\item[\textbf{(ii)}]  For this general model reduction problem we propose several numerical optimization algorithms.  The first method is  using a  gradient update for solving  the KKT system, leading to a simple iteration involving only matrix multiplications. However, with this update the stability of the approximation is achieved asymptotically. The second solution is based on a partial minimization approach. We show that for the evaluation of the gradient of the objective function we need to solve two Lyapunov  equations associated to the Gramians, but the gradient is Lipschitz continuous. Therefore,  a gradient-based algorithm is developed, ensuring convergence due to the smoothness of the objective function.  Although the gradient evaluation is expensive, each iteration provides a stable reduced order model, whereas the first method yields a stable reduced order model only asymptotically. Finally, we propose a convex SDP relaxation of the original optimization problem and derive sufficient conditions when this relaxation is exact. Note that the interpolation points obtained are the spectrum of a squared matrix computed by each of these algorithms.  

\item[\textbf{(iii)}]  We also consider a relaxed version of the general model reduction problem, searching only for the free parameters that yield the optimal reduced order model from the family of models matching $\nu$  moments at \emph{fixed} interpolation points. Optimization formulations for this particular problem are also proposed and subsequently the previous numerical optimization algorithms can be also applied  to solve  this simpler problem, obtaining   similar convergence guarantees. Finally,  we illustrate the efficiency of our results numerically using several test problems. 
\end{enumerate}

\medskip

\noindent \textit{Content}.  The paper is organized as follows. In Section \ref{sect_prel} we briefly  review the main results for Sylvester equation-based time-domain moment matching model reduction framework for linear time-invariant  systems.  In Section \ref{sect_H2formulation}, we formulate two optimal $H_2$-norm model reduction problems, recast them as  optimization problems with a Gramian-based cost function, and derive the corresponding first-order optimality conditions. We also analyze several numerical optimization methods for solving these problems. Finally,  in Section \ref{sect_expl} we illustrate the efficiency of our theory on several test examples.

%%%%%%%%%%%%%%%%%%%%%%%%%%%%%%%%%%%%%%%%%%%%

\section{Preliminaries}
\label{sect_prel} 
In this section we briefly review the main results for  Sylvester equation-based time-domain moment matching model reduction for linear time invariant  systems, see also \cite{astolfi-TAC2010,i-astolfi-colaneri-SCL2014}.  To this end, we review the notions of controllability and the observability Gramians for a linear system and the computation of the associated  $H_2$-norm based on the Gramians. 

%%%%%%%%%%%%%%%%%%%%%%%%%%%%%%%%%%%%%%%%%%%%%

\subsection{Linear  systems}
Consider a linear time invariant (LTI), minimal, square, dynamical  system:
\begin{eqnarray}
\label{system}
\Sigma: \quad \dot x &=& Ax+Bu, \quad  y=  Cx,
\end{eqnarray}
with the state $x\in\mathbb{R}^n$, the input $u\in\mathbb{R}^m$ and the output $y\in\mathbb R^p$. The transfer function of \eqref{system} is:
\begin{equation}\label{tf}
K(s)=C(sI-A)^{-1}B,\quad K:\mathbb{C} \to \mathbb {C}^{p \times m}.
\end{equation}
Throughout the rest of the paper we assume that the system \eqref{system} is stable, i.e.,  $\sigma(A) \subset \mathbb{C}^-$.

%%%%%%%%%%%%%%%%%%%%%%%%%%%%%%%%%%%%%%%%%%%%%%%%%%%%%%%%%%%%%%%%%%%%%%%%%

\subsection{Sylvester equation-based moment matching}
Assume that \eqref{system} is a minimal realization of the transfer function
$K(s)$.  The moments of (\ref{tf}) are  defined as follows:
\begin{definition}
\label{def_moment}
\cite{antoulas-2005,astolfi-TAC2010}
%The 0-moment of system (\ref{system}) at $s_{1}\in\mathbb{C}$, along direction $\ell\in\mathbb C^m$ is the complex matrix $\eta_{0}(s_{1})=C(s_{1}I-A)^{-1}B\ell\in\mathbb{C}^p$. 
The $k$-moment of (\ref{system}) at $s_{1}$, along direction $\ell\in\mathbb C^m$ is: 
\[ \eta_{k}(s_{1})=\frac{(-1)^{k}}{k!}\frac{\opd^{k}K(s)}{\opd s^{k}}\ell\in\mathbb C^{p},\quad k\geq 0. \]
 %Dually, the 0-moment of system (\ref{system}) at $s_{1}\in\mathbb{C}$, along direction $r\in\mathbb C^{1\times p}$ is the complex matrix $\eta_{0}(s_{1})=rC(s_{1}I-A)^{-1}B\in\mathbb{C}^{1\times p}$. 
The $k$-moment of system (\ref{system}) at $s_{1}$, along direction $r\in\mathbb C^{1\times p}$ is: 
\[ \eta_{k}(s_{1})=r\frac{(-1)^{k}}{k!}\frac{\opd^{k}K(s)}{\opd s^{k}}\in\mathbb C^{1\times p},\quad k\geq 0. \] 
\end{definition}

\noindent Consider the linear system (\ref{system}) and let
the matrices $S\in\mathbb{R}^{\nu\times\nu}$, $L=[\ell_1\ \ell_2\ ...\ \ell_l]\in\mathbb{C}^{m\times\nu}$, where $\ell_i\in\mathbb C^m$,
and  $Q\in\mathbb{R}^{\nu\times\nu}$, $R=[r_1^*\ ...\ r_\ell^*]^*\in\mathbb{C}^{\nu\times p}$,  where $r_i\in\mathbb{C}^{1\times p}$, be such that the pair $(L,S)$ is observable and  $(Q,R)$ is controllable, respectively. Let $\Pi\in\mathbb{R}^{n\times\nu}$ and $\Upsilon\in\mathbb{R}^{\nu\times n}$ be the solutions of  Sylvester equations:
\begin{subequations}
\label{eq_S}
\begin{eqnarray}
A\Pi+BL & = & \Pi S,
\label{eq_Sylvester_Pi}\\
\Upsilon A+RC & = &  Q\Upsilon,
\label{eq_Sylvester_Y}
\end{eqnarray}
\end{subequations} respectively. Furthermore, since the system is minimal, assuming that $\sigma(A)\cap\sigma(S)=\emptyset$, then
$\Pi$ is the unique solution of the equation (\ref{eq_Sylvester_Pi})
and ${\rm rank}\ \Pi=\nu$. Assuming that $\sigma(A)\cap\sigma(Q)=\emptyset$,
then $\Upsilon$ is the unique solution of the equation (\ref{eq_Sylvester_Y})
and ${\rm rank}\ \Upsilon=\nu$,  see, e.g., \cite{desouza-bhattacharyya-LAA1981}. Then, the moments of  system (\ref{system}) are characterized as follows:
\begin{proposition}
\label{prop_mom_time}
\cite{astolfi-TAC2010,astolfi-CDC2010}\label{def_PI}
\begin{enumerate}[wide]
\item The moments of system (\ref{system}) at the interpolation points
$\{s_{1},s_{2},...,s_{\ell}\}=\sigma(S)$ are determined  by the elements of the matrix $C\Pi$.
\item The moments of system (\ref{system}) at the interpolation points
$\{s_{1},s_{2},...,s_{\ell}\}=\sigma(Q)$ are determined by the elements of the matrix $\Upsilon B$.
\end{enumerate}
\end{proposition}

\noindent The next result gives necessary and sufficient conditions
for a low order system to achieve moment matching:
\begin{proposition}
\cite{astolfi-TAC2010,astolfi-CDC2010}
\label{prop_FGL}
Consider the reduced order  system:
\begin{equation}
\dot{\xi}=F\xi+Gu, \quad \psi=H\xi,
\label{red_mod_F}
\end{equation}
with $F\in\mathbb{R}^{\nu\times\nu}, \; G\in\mathbb{R}^{\nu\times m}, \; H\in\mathbb{R}^{p\times\nu}$ and the corresponding transfer function:
\begin{equation}
\label{eq_tf_F}
\widehat K(s)=H(sI-F)^{-1}G.
\end{equation}
Let $S\in\mathbb{C}^{\nu\times\nu}$ and $L\in\mathbb{C}^{m\times\nu}$ be  such that the pair $(L,S)$ is observable, and let $Q\in\mathbb{C}^{\nu\times\nu}$ and $R\in\mathbb{C}^{\nu\times p}$ be  such that the pair $(Q,R)$ is controllable. Moreover, assume that $\sigma(S)\cap\sigma(A)=\emptyset$ and $\sigma(Q)\cap\sigma(A)=\emptyset$. Then, the following statements hold:
\begin{enumerate}
\item Assume that $\sigma(F)\cap\sigma(S)=\emptyset$. Then, the system (\ref{red_mod_F})
matches the moments of the original system \eqref{system} at $\sigma(S)$ if and only if:
\begin{equation}
\label{eq_MM_CPi}
HP=C\Pi,
\end{equation}
where the invertible matrix $P\in\mathbb{C}^{\nu\times\nu}$ is the
unique solution of the  Sylvester equation $$FP+GL=PS.$$
\item Assume that $\sigma(F)\cap\sigma(Q)=\emptyset$. Then, the system (\ref{red_mod_F})
matches the moments of the original system  \eqref{system}  at $\sigma(Q)$ if and only if:
\begin{equation}
\label{eq_MM_YB}
\Upsilon B=PG,\
\end{equation}
where the invertible matrix $P\in\mathbb{C}^{\nu\times\nu}$ is the
unique solution of the Sylvester equation $$QP=PF+RH.$$ 
\end{enumerate}
\end{proposition}

\noindent We are now ready to present families of $\nu$ order models that match $\nu$ moments of the given system \eqref{system}:
\begin{enumerate}[wide]
\item[\textbf{(I)}] The approximation, parameterized in the interpolations points given by the spectrum of $S$ and the free  parameters given by $G$ and $L$  
\begin{equation}\label{redmod_CPi}
\widehat \Sigma_{(S,G,L)}: \quad  \dot{\xi}=(S-GL)\xi+Gu, \quad  \psi=C\Pi\xi,
\end{equation}
with the transfer function
\begin{equation}\label{tf_redmod_CPi}
\widehat K(s)=C\Pi(sI-S+GL)^{-1}G,
\end{equation}
describes a family of $\nu$ order models  that achieve moment matching at $\sigma(S)$ satisfying the following properties and constraints:
\begin{enumerate}
\item $\widehat \Sigma_{(S,G,L)}$ is parameterized in the triplet $(S,G,L)$, with $S \in \mathbb{C}^{\nu \times \nu}$, $G \in \mathbb{C}^{\nu\times m}$ and $L \in\mathbb{C}^{m \times \nu}$ such that the pair $(L,S)$ is observable
\item $\sigma(S)\cap\sigma(A)=\emptyset$
\item $\sigma(S-GL) \cap \sigma(S)=\emptyset$.
\end{enumerate}
If the pair of observable matrices $(L,S)$ is a priori fixed, and consequently $\nu$ interpolation points in $\sigma(S)$ are fixed,  then the system $\widehat{\Sigma}_G$ from  \eqref{redmod_CPi}  defines a family of $\nu$ order models that match $\nu$ moments along  directions $\ell_i$ of the original system  \eqref{system} at $\sigma(S)$  and satisfies the following properties and constraints:
\begin{enumerate}
\item $\widehat{\Sigma}_G$ is parametrized in  free parameters  $G$
\item $\sigma(S-GL) \cap \sigma(S)=\emptyset$.
\end{enumerate}	

\item[\textbf{(II)}]  Similarly, the approximation, parameterized in the interpolations points given by the spectrum of $Q$ and the free  parameters given by $H$ and $R$
\begin{equation}
\label{redmod_YB}
\bar \Sigma_{(Q,H,R)}:\quad \dot{\xi}=(Q-RH)\xi+\Upsilon Bu, \quad  \psi=H\xi,
\end{equation}
with the transfer function
\begin{equation}\label{tf_redmod_YB}
\bar K(s)=H(sI-S+GL)\Upsilon B,
\end{equation}
describes a family of $\nu$ order models  that achieve moment matching at $\sigma(Q)$ satisfying the following properties and constraints:
\begin{enumerate}
\item $\widehat \Sigma_{(Q,R,H)}$ is parameterized in the triplet $(Q,R,H)$, with $Q \in \mathbb{C}^{\nu \times \nu}$, $H^T \in \mathbb{C}^{p\times \nu}$ and $R \in\mathbb{C}^{\nu\times p}$ such that the pair $(Q,R)$ is controllable
\item $\sigma(Q)\cap\sigma(A)=\emptyset$
\item $\sigma(Q-RH) \cap \sigma(Q)=\emptyset$.
\end{enumerate}
If the pair of controllable matrices $(Q,R)$ is a priori fixed, then $\bar{\Sigma}_H$  yielded by \eqref{redmod_YB} defines a family of $\nu$ order models that match $\nu$ moments along prescribed directions $r_i$ of \eqref{system} at $\sigma(Q)$ fixed, satisfying the following properties and constraints:
\begin{enumerate}
\item $\bar{\Sigma}_H$ is parametrized in $H$
\item $\sigma(Q-RH) \cap \sigma(Q)=\emptyset$.
\end{enumerate}
\end{enumerate}

%%%%%%%%%%%%%%%%%%%%%%%%%%%%%%%%%%%%%%%

\subsection{Computation of moments} 
In practice, the moments $C\Pi$ and $\Upsilon B$ are \emph{not} computed solving the Sylvester equation \eqref{eq_S}, but using Krylov projections. In this section we recall two different notions of moments based on Krylov projections. This definition allows for development of efficient numerical algorithms for the computation of reduced order models, i.e., the Lanczos procedures, see, e.g., \cite{devillemagne-skelton-IJC1987,feldman-freund-1995,gallivan-grimme-vandooren-NA1996,grimme-1997,jaimoukha-kasenally-SIAM1997,gugercin-antoulas-2004} and references therein. These algorithms achieve moment matching through iterative procedures. As presented in \cite{i-astolfi-colaneri-SCL2014}, given a set of points in the complex plane, not among the poles of the given system, Krylov projections may be constructed. In particular, let $s_1,s_2,...,s_\nu,s_{\nu+1},s_{\nu+2},...,s_{2\nu}\in\mathbb{C}\setminus\sigma(A)$, $s_i\ne s_j,\ i\ne j$ and let $V\in\mathbb{C}^{n\times \nu}$ and $W\in\mathbb{C}^{n\times \nu}$ be, respectively:
\begin{subequations}
\begin{eqnarray}\label{projection_W_V}
V &=& [(s_1I-A)^{-1}B\quad (s_2I-A)^{-1}B\quad \dots\quad (s_{\nu}I-A)^{-1}B], \label{projection_V} \\
W &=& [ (s_{\nu+1}I-A^*)^{-1}C^* \quad (s_{\nu+2}I-A^*)^{-1}C^* \quad \dots \quad (s_{2\nu}I-A^*)^{-1}C^*]. \label{projection_W}
\end{eqnarray}
\end{subequations}

\noindent The next result follows from Definition \ref{def_moment}, writing the moments at each point $s_i$ in matrix form:
\begin{proposition}
\cite{i-astolfi-colaneri-SCL2014}
\label{def_VW}
The moments of system (\ref{system}) at $s_1,s_2,...,s_\nu \notin\sigma(A)$ are the elements of the matrix $CV$. We call $V$ the right Krylov projection matrix. Furthermore, the moments of system (\ref{system}) at $s_{\nu+1},s_{\nu+2},...,s_{2\nu}\notin\sigma(A)$ are the elements of the matrix $WB$. We call $W$ the left Krylov projection.
\end{proposition}

\noindent In the sequel, we briefly overview the equivalent relation between the moments described in Proposition \ref{prop_mom_time} and the moments described by Proposition \ref{def_VW}. In \cite{i-astolfi-colaneri-SCL2014,astolfi-CDC2010} relations between the projections $V$ and $W$ and the solutions of the Sylvester equations $\Pi$ and $\Upsilon$ were established:
\begin{lemma}
\label{lemma_V_and_Pi}
\cite{astolfi-CDC2010} 
\begin{enumerate}[wide]
\item Let $\Pi$ be the solution of the Sylvester equation (\ref{eq_Sylvester_Pi})
and let the projector $V$ be as in (\ref{projection_V}).
Then, there exists a square, non-singular, matrix $T \in\mathbb{C}^{\nu \times \nu}$ such that $\Pi = V T$. For $T=I_\nu$, $V$ from \eqref{projection_V} is the unique solution of equation \eqref{eq_Sylvester_Pi} for $S=\diag (s_1,s_2,\dots,s_\nu)$ and $L=[\ell_1\ \ell_2\ ...\ \ell_l]\in\mathbb{R}^{m\times\nu}$.

\item Let $\Upsilon$ be the solution of the Sylvester equation (\ref{eq_Sylvester_Y})
and let the projector $W$ be as in (\ref{projection_W}). Then, there exists a square, non-singular, matrix $T \in\mathbb{C}^{\nu \times \nu}$ such that $\Upsilon = T W$. For $T=I_\nu$, $W$ from \eqref{projection_W} is the unique solution of equation \eqref{eq_Sylvester_Y}  for $Q=\diag(s_{\nu+1},s_{\nu+2},\dots,s_{2n})$ and $R=[r_1^*\ ...\ r_\ell^*]^*\in\mathbb{R}^{\nu\times p}$. 
\end{enumerate}
\end{lemma}

\noindent Hence, the moments of system \eqref{system} at $\sigma (S)$ and/or $\sigma (Q)$ as in Proposition \ref{prop_mom_time}, are computed as follows:

\begin{cor}\label{cor_momPiY_Krylov}
Consider system \eqref{system}. Let $(L,S)$ be a pair of observable matrices of appropriate dimensions and let $(Q,R)$ be another pair of controllable matrices of appropriate dimension, respectively, such that $\sigma(S)\cap\sigma(Q)$. Then:
\begin{enumerate}[label=\roman*),wide]
\item the moments of system \eqref{system} at $\sigma(S)$ are given by $C\Pi=CVT$, where $\Pi$ is the unique solution of the Sylvester equation \eqref{eq_Sylvester_Pi} and $V$ is given by \eqref{projection_V}.
\item the moments of system \eqref{system} at $\sigma(Q)$ are given by $\Upsilon B=TWB$, where $\Upsilon$ is the unique solution of the Sylvester equation \eqref{eq_Sylvester_Y} and $W$ is given by \eqref{projection_W}. 
\end{enumerate}
\end{cor}

\noindent The results of Proposition \ref{def_VW}, Lemma \ref{lemma_V_and_Pi}, and Corollary \ref{cor_momPiY_Krylov} also hold for higher order moments at a set of interpolation points $s_1,\dots,s_l\in\mathbb{C}$ which are not poles of the given transfer function $K$. Let $s_i$, $i=0,...,l$ and $l\geq 0$. To this end, take $j_i\geq 0$ such that:

\begin{equation}\label{eq_mu}
\sum_{i=0}^l (j_i -1)=\nu.
\end{equation}
For each $i$, let $\eta_0(s_i),...,\eta_{j_i}(s_i)$ denote the first $j_i+1$ moments of the system  defined by \eqref{system} at the given points $s_i$. Then, these moments are characterized by the matrix $CV$, with:
\begin{align*}
V&=[V_0\ell_0\ \dots\ V_l\ell_l]\in\mathbb{C}^{n\times \nu},\\
V_i&=[(s_iI-A)^{-1}B\ (s_iI-A)^{-2}B\ \dots\ (s_iI-A)^{-j_i}B]\in\mathbb{C}^{n\times j_i}.
\end{align*}
Furthermore, $V$ is the solution of the Sylvester equation \eqref{eq_Sylvester_Pi}, for $S=\diag(\Sigma_0,\ \dots\ \Sigma_l),$ with $\Sigma_i\in\mathbb{C}^{j_i\times j_i}$ a Jordan block matrix of the eigenvalue $s_i$ with multiplicity $j_i$ and $L=[\ell_0\ \dots\ \ell_l]\in\mathbb{R}^{1\times \nu}$, with $\ell_i=[1\ 0\ \dots\ 0]\in\mathbb{R}^{1\times j_i}$. The results follow directly from the arguments used in \cite{astolfi-TAC2010,i-astolfi-colaneri-SCL2014}. Note that the moment matching conditions {\eqref{eq_MM_CPi} and \eqref{eq_MM_YB}} are equivalent  \cite{gallivan-vandendorpe-vandooren-SIAM2004}, up to a constant coordinate transformation to the right tangential interpolation conditions:
\begin{align*}
K(s_i)\ell_{j_i}  &= \widehat{K}(s_i)\ell_{j_i},\quad K'(s_i)\ell_{j_i}  = \widehat{K}'(s_i)\ell_{j_i}, \\
                   & \cdots \\
\frac{\text{d}^{j_i}}{\text{d}s^{j+i}}K(s)\ell_{j_i} &= \frac{\text{d}^{j_i}}{\text{d}s^{j+i}}\widehat K(s)\ell_{j_i},\quad i=0:l
\end{align*}
and the left tangential interpolation conditions:
\begin{align*}
r_{j_i}K(s_i)  &= r_{j_i}\hat{K}(s_i),\quad r_{j_i}K'(s_i) = r_{j_i}\hat{K}'(s_i), \\
                   & \cdots \\
r_{j_i}\frac{\text{d}^{j_i}}{\text{d}s^{j+i}}K(s) &= r_{j_i}\frac{\text{d}^{j_i}}{\text{d}s^{j+i}}\widehat K(s),\quad i=0:l.
\end{align*}
Note that these reduced order models are parameterized in $L$ and $R$, respectively. Their choice is important for computing subfamilies of models that preserve specific properties, for establishing appropriate directions for interpolation and for finding the most accurate approximants. Note that we can also do moment matching with prescribed poles. The poles of the reduced order model may be placed, for example, in the the open left half plane, by properly selecting $G$ or $H$ respectively, yielding the subfamily of stable reduced order models that match the moments of \eqref{system} at $\sigma(S)$ or $\sigma(Q)$, respectively.
\begin{proposition}
\cite{astolfi-TAC2010}
\label{prop_assign_poles} 
Consider an LTI system \eqref{system}. Furthermore, consider the families of reduced order models $\widehat\Sigma$ and $\bar\Sigma$ that match the moments of \eqref{system} at $\sigma(S)$ and $\sigma(Q)$, respectively. Let $\lambda_i\in\mathbb{C}$, $i=1,...,\nu$, be such that $\lambda_i\notin \sigma(S)$ or $\lambda_i\notin \sigma(Q)$. Then:
\begin{enumerate}
\item There exists a subfamily of models of the form $\widehat\Sigma$, with the property that the spectrum of each model contains $\lambda_1,...,\lambda_\nu$, i.e., there exists $G$ such that $\{\lambda_1,...,\lambda_\nu\}=\sigma(S-GL)$.
\item There exists a subfamily of models of the form $\bar\Sigma$, with the property that the spectrum of each model contains $\lambda_1,...,\lambda_\nu$, i.e., there exists $H$ such that $\{\lambda_1,...,\lambda_\nu\}=\sigma(Q-RH)$.
\end{enumerate}
\end{proposition}
 
\noindent Based on the previous discussion in the following we make the following working assumption: Matrices $\Pi$ and $\Upsilon$, unique solutions of \eqref{eq_S} are formed using Krylov projections $V$ and $W$ in \eqref{projection_W_V}, respectively, by Lemma \ref{lemma_V_and_Pi}. Furthermore, the moments of system \eqref{system} at $\sigma(S)$ and at $\sigma(Q)$ are computed efficiently based on Corollary \ref{cor_momPiY_Krylov}, respectively. It is not required to explicitly solve equations \eqref{eq_S}.

%%%%%%%%%%%%%%%%%%%%%%%%%%%%%%%%%%%%%%%%
%%%%%%%%%%%%%%%%%%%%%%%%%%%%%%%%%%%%%%%%%

\subsection{$H_2$-norm based on the Gramians of linear systems}
Let us also  briefly recall the definition of the $H_2$-norm of an LTI system and its computation based on the controllability and the observability Gramians, respectively. Given the LTI system \eqref{system}, the controllability Gramian $W$ and the observability
Gramian $M$ are the solutions of the following Lyapunov  equations \cite{antoulas-2005}:
\begin{subequations}\label{eq_Gramians}
\begin{align}
AW+WA^T+BB^T &=0, \\
A^TM+MA+C^TC &=0.
\end{align}
\end{subequations}

\noindent Let ${\cal H}_2$ denote the Hilbert space of complex functions analytic in the open right-half plane and square integrable. Note that the transfer functions $K$ and $\widehat K$ are elements of ${\cal H}_2$. By \cite{gugercin-antoulas-beattie-SIAM2008}  ${H}_2$-norm is defined as:
$$\|K\|_{{H}_2}=\sqrt{\int_{-\infty}^{\infty} |K(j\omega)|^2 \opd\omega}.$$
The following result provides a computation formula for  ${H}_2$-norm of a rational transfer function $K$.
\begin{lemma}
\label{lema_2norm_gram}
\cite{gugercin-antoulas-beattie-SIAM2008}
Consider the LTI system \eqref{system} with the transfer function \eqref{tf}. Then:
\begin{equation}\label{eq_2norm_gram}
\|K\|_{{H}_2}^2=C^TWC=B^TMB,
\end{equation}
where $W$ and $M$ the solutions of equations \eqref{eq_Gramians}.
\end{lemma}

%%%%%%%%%%%%%%%%%%%%%%%%%%%%%%%%%%%%%%%%%%%%%%%%

\section{$H_2$ model reduction by moment matching and  optimization}
\label{sect_H2formulation}
This section presents the main theoretical contribution of our paper. Based on the previous parametrizations of the reduced models and using  the $H_2$-norm of the approximation error as objective function, we write several  optimization problems  to optimally  determine the approximation yielding the minimal $H_2$-norm error.  We  also derive the KKT (optimality) conditions for this optimization problems in terms of the controllability and the observability Gramians  of the error system. Finally, based on these optimality conditions we propose several  gradient-based methods for solving the corresponding optimization problems, with mathematical guarantees on its  convergence. Due to the symmetry of families $\widehat{\Sigma}$ and $\bar{\Sigma}$, in the sequel we focus on the parametrization of reduced system $\widehat{\Sigma}$ in $(S,G,L)$. The general optimal $H_2$ model reduction problem by moment matching is formulated as follows:

\begin{prob}\label{prob_optH2_general}
Given an LTI system \eqref{system} with the transfer function $K$ given in \eqref{tf}, find a reduced order LTI system $\widehat{\Sigma}_{(S,G,L)}$ in the family  \eqref{redmod_CPi} with the transfer function $\widehat K$ defined in \eqref{tf_redmod_CPi}, given in terms of the interpolations points $\sigma(S)$ and  free parameters $L$ and $G$, that match  $\nu$ \emph{fixed} moments of \eqref{system} at $\sigma(S)$ and the following conditions are satisfied:
\begin{enumerate}[label=(\roman*),wide,nosep]
\item the $H_2$ norm of the error system  $\|K-\widehat K\|_2$ is minimal
\item the reduced model $\widehat K$ is stable, i.e. $\sigma(S-GL) \subset \mathbb C^-$
\item the pair $(S,L)$ is observable, $\sigma(S) \cap \sigma(A) = \emptyset$, and $\sigma(S) \cap \sigma(S-GL) = \emptyset$. \fin
\end{enumerate}
\end{prob}

\noindent  We can also consider a relaxed formulation of this problem.  Assuming that the pair $(L,S)$ is fixed a priori for the reduced order model \eqref{redmod_CPi}, such that $(L,S)$ is observable and $\sigma(S) \cap \sigma(A) = \emptyset$, we search for a reduced order model $\widehat{\Sigma}_G$ in the family  \eqref{redmod_CPi} parameterized only in $G$, that yields the minimal $H_2-$norm of the approximation error. Hence, we compute the best $\nu$ order model from the family $\widehat \Sigma_G$, matching fixed $\nu$ moments of \eqref{system}. Thus, we may formulate the following particular instance of Problem \ref{prob_optH2_general}, separately:
\begin{prob}
\label{prob_optH2_G}
\emph{Fix} $S\in\mathbb{C}^{\nu \times \nu}$ and $L\in\mathbb{C}^{m \times \nu}$ two matrices  such that the pair $(L,S)$ is observable and $\sigma(S) \cap \sigma(A) = \emptyset$. Given the LTI system \eqref{system} with the transfer function \eqref{tf} and the family of reduced order models $\widehat \Sigma_G$ as in \eqref{redmod_CPi} that match the $\nu$ \emph{fixed} moments of \eqref{system} at $\sigma(S)$, find the free parameters defined in terms of  matrix $G$ such that the following conditions are satisfied:
\begin{enumerate}[label=(\roman*),wide,nosep]
 \item the $H_2$ norm of the error system $\|K-\widehat K\|_2$ is minimal;
\item the reduced model $\widehat K$ is stable,  i.e., $\sigma(S-GL) \subset \mathbb C^-$;
\item $\sigma(S) \cap \sigma(S-GL) = \emptyset$. \fin
\end{enumerate}
\end{prob}

\noindent Problems \ref{prob_optH2_general} and  \ref{prob_optH2_G} can be recast  in terms of the computation of the ${\cal H}_2$ norm of the Gramians of the realization of the \emph{error system}:
$$\K = K - \widehat K,$$
with $\widehat K$ from \eqref{tf_redmod_CPi}, parameterized in $(S,G,L)$ or $G$, respectively. Let $(\A_e,\B_e,\C_e)$ be a state-space realization of the error transfer function $\K$: $$\K(s)=\C_e(sI-\A_e)^{-1}\B_e,$$ where
\begin{equation}\label{eq_err_realization}
\A_e=\begin{bmatrix} A & 0 \\ 0 & S - GL \end{bmatrix},\ \B_e=\begin{bmatrix} B \\ G \end{bmatrix},\ \C_e=C\begin{bmatrix} I &  -\Pi\end{bmatrix}.
\end{equation}

\noindent Denote the controllability and the observability Gramians of \eqref{eq_err_realization} by $\W$ and $\M$, respectively. They are solutions of the following Lyapunov equations:
\begin{subequations}\label{eq_Gramians_err}
\begin{align}
\label{W_Lyap}
\A_e\W+\W\A_e^T+\B_e\B_e^T &=0, \\
\label{M_Lyap}
\A_e^T\M+\M\A_e+\C_e^T\C_e &=0.
\end{align}
\end{subequations}

\noindent Let us also recall a standard result for Lyapunov equations:
\begin{lemma}
\label{lema:lyap_sol}
Let $\A_e$ be given stable matrix, i.e., $\sigma(\A_e) \subset \mathbb C^-$. Then, there exist unique solutions $\W$ and $\M$  positive semidefinte of \eqref{W_Lyap} and \eqref{M_Lyap}, respectively. 
\end{lemma}

\noindent  Below, we partition  $\W$ and $\M$ following the block structure of matrix $\A_e$:
\begin{equation}
\W=\begin{bmatrix} W_{11} & W_{12} \\ W_{12}^T & W_{22}\end{bmatrix},\;\;   \M=\begin{bmatrix} M_{11} & M_{12} \\ M_{12}^T & M_{22}\end{bmatrix}.
\end{equation}

%%%%%%%%%%%%%%%%%%%%%%%%%%%%%%%%%%%%%%%%%%%%%%%%%%%%%%%%%

\subsection{Optimization formulation of Problem \ref{prob_optH2_general}}
In this section we propose an optimization formulation for the general Problem \ref{prob_optH2_general}, where recall that the parametrisation of the reduced order model is done through matrices $S, G$ and $L$. Let us define the feasible set for the reduced model:
\[  {\cal R}= \left\{ (S, G, L): \; (S,L) \ \text{obs.}, \; \sigma(S-GL) \subset \mathbb C^-, \; \sigma(S) \cap \sigma(A) = \emptyset, \; \text{and} \; \sigma(S) \cap \sigma(S-GL) = \emptyset \right\}.  \]
By \eqref{eq_2norm_gram}, the Problem  \ref{prob_optH2_general} becomes:
\begin{align}\label{eq_2norm_gram_err_general0}
\min_{(S,G,L) \in {\cal R}} \|\K\|_2^2 & =
\min_{(S,G,L) \in {\cal R}, \ \W \ \text{s.t.}\; \eqref{W_Lyap}} C \begin{bmatrix} I &  -\Pi\end{bmatrix}\begin{bmatrix} W_{11} & W_{12} \\ W_{12}^T & W_{22}\end{bmatrix}\begin{bmatrix} I \\  -\Pi^T\end{bmatrix}C^T \nonumber\\
&= \min_{(S,G,L) \in {\cal R},\  \M \ \text{s.t.} \; \eqref{M_Lyap}} \begin{bmatrix} B^T & G^T \end{bmatrix} \begin{bmatrix} M_{11} & M_{12} \\ M_{12}^T & M_{22}\end{bmatrix} \begin{bmatrix} B \\ G \end{bmatrix}.\nonumber
\end{align}
We now consider the problem formulation in terms of the observability Gramian $\M$, written explicitly in matrix form as:
\begin{align}
\label{eq_2norm_gram_err_general_m0}
& \min_{(S,G,L,\M,\Pi)} \text{Trace}(\B_e^T \M \B_e) \\
& \text{s.t.:} \;(S,L) \ \text{observable},\;  \sigma(S) \cap \sigma(A) = \emptyset, \; \text{and} \; \sigma(S) \cap \sigma(S-GL) = \emptyset, \nonumber\\
& \qquad A \Pi + B L = \Pi S,  \; \sigma(S-GL) \subset \mathbb C^-, \; \A_e^T \M + \M \A_e + \C_e^T \C_e =0. \nonumber
\end{align}
Note that, since $\sigma(S)\cap\sigma(A)=\emptyset$, the Sylvester equation \eqref{eq_Sylvester_Pi} has a unique solution $\Pi =\Pi(S,L)$. However, by our working  assumptions, \eqref{eq_Sylvester_Pi} need \emph{not} be solved since, by Lemma \ref{lemma_V_and_Pi}, we can take $\Pi=VT$, with $V$ from \eqref{projection_V} and $T$ some non-singular matrix. Therefore, Problem \ref{prob_optH2_general} can be reformulated equivalently as:
\begin{align}
\label{eq_2norm_gram_err_general_m00}
& \min_{(S,G,L,\M)} \text{Trace}(\B_e^T \M \B_e) \\
& \text{s.t.:} \; (S,L) \ \text{observable}, \; \sigma(S) \cap \sigma(A) = \emptyset, \; \text{and} \; \sigma(S) \cap \sigma(S-GL) = \emptyset, \nonumber\\
& \qquad   \sigma(S-GL) \subset \mathbb C^-, \;\; \A_e^T \M + \M \A_e + \C_e^T \C_e =0,  \nonumber
\end{align}
with $\Pi=VT$,  for $V$ as in \eqref{projection_V} and $T$ some fixed non-singular matrix.    However,  it is difficult to deal with the restrictions (iii) in Problem \ref{prob_optH2_general}, i.e., the constraints $(S,L)$ observable,  $\sigma(S) \cap \sigma(A) = \emptyset$, and $\sigma(S) \cap \sigma(S-GL) = \emptyset$. Hence,  we can consider a triplet $(S,L,G)$ such  that the  pair $(S,L)$ is observable.  Then, the unknowns are the diagonal matrix $S$ and the vector $G$, while $L$ is fixed and $\Pi=VT$.  For example,  without loss of generality, we can consider a canonical form for the triplet $(S,L,G)$ as:
\[ S = \diag(s_1,s_2,\dots,s_\nu), \; L=[\ell_1 \ \ell_2 \dots \ell_\nu] \in \mathbb{R}^{m \times \nu} \; \text{and} \; G \in \mathbb{R}^{\nu \times m}, \]
such  that the  pair $(S,L)$ is automatically observable, provided that we choose $\ell_i \not = 0$,  regardless of the values of $s_i$ for all $i$. In this case we take $\Pi = V$. Moreover, the constraints $\sigma(S) \cap \sigma(A) = \emptyset$ and $\sigma(S) \cap \sigma(S-GL) = \emptyset$ are not imposed in the numerical algorithms and are usually checked at the solution of the problem.  Therefore, in the next sections we  provide several numerical methods for solving the general non-convex optimization problem \eqref{eq_2norm_gram_err_general_m00}, with unknowns   $\M, S$ and $G$, while $L=[\ell_1 \ \ell_2 \dots \ell_\nu]$ is fixed a priori.  Moreover, following our discussion above,  the constraints $(S,L)$ observable,  $\sigma(S) \cap \sigma(A) = \emptyset$ and $\sigma(S) \cap \sigma(S-GL) = \emptyset$ are also removed. In this case from \eqref{eq_2norm_gram_err_general_m00} we get the simplified non-convex optimization formulation for Problem \ref{prob_optH2_general} analyzed in the sequel:
\begin{align}
\label{eq_2norm_gram_err_general_m}
& \min_{(S,G,\M)} \text{Trace}(\B_e^T \M \B_e) \\
& \text{s.t.:} \;  \sigma(S-GL) \subset \mathbb C^-, \;\; \A_e^T \M + \M \A_e + \C_e^T \C_e =0.  \nonumber
\end{align}
Note that any  (local) solution  of optimization problem \eqref{eq_2norm_gram_err_general_m} satisfying  the constraints $(S,L)$ observable,  $\sigma(S) \cap \sigma(A) = \emptyset$ and $\sigma(S) \cap \sigma(S-GL) = \emptyset$ is also a local/global solution of Problem \ref{prob_optH2_general}.  How to efficiently numerically tackle  the aforementioned constraints (i.e.,  non-convex problem \eqref{eq_2norm_gram_err_general_m00}) so that to easily include them in an optimization algorithm remains an open question that will be investigated in the future. Let:
\begin{align} 
\label{eq:notation}
{\cal X}=\begin{bmatrix} S & G\end{bmatrix} \in \mathbb{R}^{\nu \times (\nu +m)},\quad {\cal L}=\begin{bmatrix} I_\nu \\ -L\end{bmatrix} \in \mathbb R^{(\nu+m) \times \nu} \quad \text{and}\quad {\cal E}=\begin{bmatrix}0_{\nu \times m} \\  I_m \end{bmatrix} \in \mathbb R^{(\nu+m) \times m},
\end{align} 
yielding
\begin{align}
\label{eq_XY}
G &= {\cal X}  {\cal E} \;\; \text{and} \;\; S-GL = {\cal XL}.
\end{align}
In order to clearly see the dependence on ${\cal X}$,  let us also define:
\[ {\cal A}({\cal X}) = \A_e= \begin{bmatrix} A & 0 \\ 0 & {\cal XL} \end{bmatrix}, \;\;   {\cal B}({\cal X}) = \B_e\B_e^T= \begin{bmatrix} B \\ {\cal X}  {\cal E} \end{bmatrix} \begin{bmatrix} B \\ {\cal X}  {\cal E}  \end{bmatrix}^T, \;\; {\cal C} = \C_e^T  \C_e.  \]
In the next sections we present several (equivalent) reformulations of the nonconvex problem \eqref{eq_2norm_gram_err_general_m}, accompanied by their first-order  optimality conditions.

%%%%%%%%%%%%%%%%%%%%%%%%%%%%%%%%%%%%%%%%%%

\subsubsection{ KKT approach}
Recall that our goal is to find a (local) minimum point of the non-convex problem \eqref{eq_2norm_gram_err_general_m}.  However, for a non-convex problem a minimum point is among the  KKT points, i.e. it satisfies the KKT system.  Using $\text{Trace}(MN) = \text{Trace}(NM)$ for any matrices $M, N$ of compatible sizes, in the sequel we derive the KKT system for  the non-convex problem \eqref{eq_2norm_gram_err_general_m}, which in compact form, in terms of  ${\cal X}$, can be written as follows: 
\begin{align}
\label{eq_2norm_gram_err_general_mm}
& \min_{(\M,{\cal X})} \text{Trace}( \M {\cal B}({\cal X})) \\
& \text{s.t.:} \; \;  {\cal X} \in {\cal D}_L, \;\; {\cal A}^T({\cal X})  \M + \M {\cal A}({\cal X})  + {\cal C} =0, \nonumber
\end{align}
where the open set  ${\cal D}_L =\{  {\cal X}:  \;  \sigma({\cal X L})  \subset \mathbb C^- \}$ and recall that $L$ is fixed a priori.  The Lagrangian function associated to problem  \eqref{eq_2norm_gram_err_general_m}, or equivalently \eqref{eq_2norm_gram_err_general_mm}, is given by:
\begin{equation}
\label{eq_Lagrange_KKT}
\Gamma(\W, \M, {\cal X}) = \text{Trace}( \M {\cal B}({\cal X}) ) + \text{Trace} (\W ( {\cal A}^T({\cal X}) \M + \M {\cal A}({\cal X}) + {\cal C} )),
\end{equation}
where the multiplier $\W$ is associated to the equality constraint in  \eqref{eq_2norm_gram_err_general_mm}.  Then, we write the optimization problem  \eqref{eq_2norm_gram_err_general_mm} into the max-min form:
\begin{equation}
\label{eq_2norm_gram_err_G_Lagrange_KKT}
\max_{\W}  \min_{\M, {\cal X} \in {\cal D}_L}  \Gamma(\W, \M, {\cal X}).
\end{equation}
From standard optimization arguments we know that for any solution (also called KKT or saddle point)  $(\W, \M, {\cal X})$ of problem \eqref{eq_2norm_gram_err_G_Lagrange_KKT}, we have that  $(\M, {\cal X})$ is a (possibly local) minimum point of the original problem \eqref{eq_2norm_gram_err_general_mm} \cite{Nes:04}. Moreover,  if $(\W, \M, {\cal X})$ is a solution of problem \eqref{eq_2norm_gram_err_G_Lagrange_KKT} with ${\cal X} \in {\cal D}_L$, then it must satisfy the KKT system:
\[ \nabla  \Gamma (\W, \M, {\cal X}) =0 \quad \iff \quad  
\begin{cases}
\nabla_{\W}  \Gamma (\W, \M, {\cal X}) =0\\
\nabla_{(\M, {\cal X})}  \Gamma (\W, \M, {\cal X}) =0.
\end{cases}
\]

\noindent The next theorem provides the explicit form of  the  \emph{KKT system}:
\begin{theorem}
The KKT system of optimization problem  \eqref{eq_2norm_gram_err_general_mm} is given by:
\begin{equation}
\label{eq_KKT_cond_explicit}
\nabla  \Gamma (\W, \M, {\cal X}) = 0 \; \iff \; 
\begin{cases}
{\cal A}^T({\cal X}) \M +  \M  {\cal A}({\cal X})  + {\cal C}  = 0\\
{\cal A}({\cal X}) \W +  \W  {\cal A}^T({\cal X}) + {\cal B} ({\cal X}) =0 \\
M_{12}^T B {\cal E}^T   +     M_{22} {\cal X} {\cal E} {\cal E}^T +  M_{12}^T W_{12}  {\cal L}^T +    M_{22}  W_{22} {\cal L}^T  = 0.
\end{cases}
\end{equation}
\end{theorem}

\begin{proof}
Note that the KKT system has the form:
\[
\nabla  \Gamma (\W, \M, {\cal X}) = 
\begin{bmatrix}
\nabla_\W \Gamma (\W, \M, {\cal X}) \\
\nabla_\M \Gamma (\W, \M, {\cal X})  \\ 
\nabla_{\cal X} \Gamma (\W, \M, {\cal X}) 
\end{bmatrix}
=
\begin{bmatrix}
{\cal A}^T({\cal X}) \M +  \M  {\cal A}({\cal X})  + {\cal C}  \\
{\cal A}({\cal X}) \W +  \W  {\cal A}^T({\cal X}) + {\cal B} ({\cal X}) \\
\nabla_{\cal X} \Gamma (\W, \M, {\cal X})
\end{bmatrix}.
\]
It remains to explicitly compute $ \nabla_{\cal X} \Gamma (\W, \M, {\cal X})$.  However, to write  the gradient expression, we  introduce a gradient of $\Gamma$ w.r.t. ${\cal X}$ using the trace of a matrix: 
\[    \Gamma' (\W, \M, {\cal X}) \opd{\cal X}  =    \text{Trace}( \nabla_{\cal X}^T \Gamma (\W, \M, {\cal X}) \opd{\cal X}), \quad \text{with} \; \opd{\cal X} \in  \mathbb R^{\nu \times (\nu+m)}. \]
Then, we have:
\begin{align*}
&\text{Trace}( \nabla_{\cal X}^T \Gamma (\W, \M, {\cal X}) \opd{\cal X}) \\
& =  \text{Trace}( \M {\cal B}' ({\cal X})  +  \W ( ({\cal A}'({\cal X}))^T \M + \M {\cal A}'({\cal X})))\\
&=  \text{Trace} \left( \M  \begin{bmatrix} 0 & B {\cal E}^T  \opd{\cal X}^T \\  \opd{\cal X} {\cal E} B^T &   \opd{\cal X}  {\cal E} {\cal E}^T {\cal X}^T   +  {\cal X} {\cal E} {\cal E}^T  \opd{\cal X}^T \end{bmatrix} + \W \begin{bmatrix} 0 & 0 \\\ 0 & \opd{\cal X} {\cal L} \end{bmatrix}^T \M + \M \begin{bmatrix} 0 & 0 \\\ 0 & \opd{\cal X} {\cal L} \end{bmatrix}  \W  \right)\\
&= 2 \text{Trace} \left(   {\cal E} B^T M_{12}  \opd{\cal X} +  {\cal E} {\cal E}^T X^T M_{22} \opd{\cal X} + {\cal L} W_{12}^T M_{12} \opd{\cal X}   +  {\cal L} W_{22} M_{22} \opd{\cal X}  \right),
\end{align*}
where in the last equality we used the block structure of $\W$ and $\M$.  Hence, we have:
\[  \nabla_{\cal X} \Gamma (\W, \M, {\cal X}) =  2  \left( M_{12}^T B {\cal E}^T   +     M_{22} {\cal X} {\cal E} {\cal E}^T +  M_{12}^T W_{12}  {\cal L}^T +    M_{22}  W_{22} {\cal L}^T  \right).   \]  
Finally, we get the KKT system from \eqref{eq_KKT_cond_explicit}. 
\end{proof}

\noindent The result of previous theorem  also yields the necessary optimality condition for the optimization problem   \eqref{eq_2norm_gram_err_general_mm} of the general model reduction Problem  \ref{prob_optH2_general}:
\begin{lemma}
If   $\M $ and ${\cal X} \in {\cal D}_L$ , where ${\cal X}=[S \; G]$, solves the optimization problem   \eqref{eq_2norm_gram_err_general_mm} corresponding to the model reduction Problem \ref{prob_optH2_general}, then there exists $\W$ such that the triplet $(\W,\M,{\cal X})$ solves the KKT system \eqref{eq_KKT_cond_explicit}.
\end{lemma}

%%%%%%%%%%%%%%%%%%%%%%%%%%%%%%%%%%%%%%%%%%%%%%

\subsubsection{ Partial minimization approach}
\label{sec:partimin}
Consider the non-convex optimization problem \eqref{eq_2norm_gram_err_general_m}, where  $L$ is fixed a priori,  $\Pi=VT$,  and  $\A_e = \A_e(S,G)$.  Then, the following partial minimization holds for \eqref{eq_2norm_gram_err_general_m}: 
\[ \eqref{eq_2norm_gram_err_general_m}  = \min_{(S,G): \sigma(S-GL) \subset \mathbb C^-}  \left(  \min_{\M: \A_e^T \M + \M \A_e + \C_e^T \C_e =0} \text{Trace}(\B_e^T \M \B_e)) \right).  \] 
However, if $S-GL$ and  $A$ are stable,  it follows from Lemma \ref{lema:lyap_sol}  that there  exists unique  $\M =\M(S,G) \succeq 0$  solution of the Lyapunov equation:
\[  \A_e^T \M + \M \A _e+ \C_e^T \C_e =0.  \]
Hence, for any pair $(S,G)$ stable, the partial minimization in $\M$ leads to an optimal value $f(S,G) = \min_{\M: \A_e^T \M + \M \A_e + \C_e^T \C_e =0} \text{Trace}(\B_e^T \M \B_e))$ which can be written explicitly as:
\[  f(S,G) = \text{Trace} \left( \begin{bmatrix} B \\ G \end{bmatrix}^T  \M(S,G) \begin{bmatrix} B \\ G \end{bmatrix} \right),  \] 
where $\M(S,G)$ is the unique solution of the Lyapunov equation:
\begin{align}
\label{lyap_eq_11}
\begin{bmatrix} A & 0 \\ 0 & S-GL \end{bmatrix}^T \M  + \M  \begin{bmatrix} A & 0 \\ 0 & S-GL \end{bmatrix} + \begin{bmatrix} C^TC & -C^T C_V \\ -C_V^T C & C_V^T C_V  \end{bmatrix} =0,
\end{align}
with $C_V = C \Pi = C V T$. Therefore,  we get the following equivalent reformulation for \eqref{eq_2norm_gram_err_general_m}:
\begin{align}
\label{eq_2norm_gram_err_general_m1}
& \min_{(S,G)} \text{Trace} \left( \begin{bmatrix} B \\ G \end{bmatrix}^T  \M(S,G) \begin{bmatrix} B \\ G \end{bmatrix} \right) \\
& \text{s.t.}: \;\; \sigma(S-GL) \subset \mathbb C^- \quad \text{and} \quad \eqref{lyap_eq_11}. \nonumber
\end{align}
Using the notation in \eqref{eq:notation}, the non-convex problem \eqref{eq_2norm_gram_err_general_m1} becomes:
\begin{align}
\label{eq_2norm_gram_err_general_m2}
& \min_{\cal X} \text{Trace} \left( \begin{bmatrix} B \\ {\cal X}  {\cal E} \end{bmatrix}^T  \M({\cal X}) \begin{bmatrix} B \\ {\cal X}  {\cal E} \end{bmatrix} \right) \\
& \text{s.t.}: \;\; \sigma({\cal X L}) \subset \mathbb C^- \quad \text{and} \quad \eqref{lyap_eq_12}, \nonumber
\end{align}
where $\M({\cal X})$ is the unique positive semidefinte solution of  the Lyapunov equation:
\begin{align}
\label{lyap_eq_12}
\begin{bmatrix} A & 0 \\ 0 & {\cal XL} \end{bmatrix}^T \M({\cal X})  + \M({\cal X})  \begin{bmatrix} A & 0 \\ 0 & {\cal XL} \end{bmatrix} + \begin{bmatrix} C^TC & -C^T C_V \\ -C_V^T C & C_V^T C_V  \end{bmatrix} =0.
\end{align}
For solving the equivalent  non-convex problem \eqref{eq_2norm_gram_err_general_m2} we can  apply any first- or second-order optimization method. For this type of optimization scheme we need to compute the gradient and even the Hessian of the objective function. In the sequel, we show that we can compute the gradient of the objective function of  \eqref{eq_2norm_gram_err_general_m2} solving two Lyapunov equations.  Indeed, by $\text{Trace}(MN) = \text{Trace}(NM)$ for any matrices $M, N$ of compatible sizes, the non-convex objective function of \eqref{eq_2norm_gram_err_general_m2} becomes in terms of  the notation \eqref{eq:notation}:
\[  f({\cal X}) =  \text{Trace} \left( \begin{bmatrix} B \\ {\cal X} {\cal E} \end{bmatrix}^T  \M({\cal X}) \begin{bmatrix} B \\ {\cal X} {\cal E} \end{bmatrix} \right)  = \text{Trace} \left(  \M({\cal X})   {\cal B}({\cal X})  \right).  \]

\begin{theorem}
The  objective function  $f$ of \eqref{eq_2norm_gram_err_general_m2}  is differentiable on the set of stable matrices ${\cal D}_L$  and the gradient of $f$ at ${\cal X} \in {\cal D}_L$ is given by: 
\begin{align}
\label{eq_gradient1}
\nabla f({\cal X}) =  2  \left[ M_{12}^T({\cal X})  W_{12}({\cal X})  {\cal L}^T +    M_{22} ({\cal X}) W_{22}({\cal X}) {\cal L}^T +   M_{12}^T({\cal X}) B {\cal E}^T   +     M_{22}({\cal X}) {\cal X} {\cal E} {\cal E}^T \right], 
\end{align}
where $\M({\cal X}) $ solves the Lyapunov equation \eqref{lyap_eq_12} and   $\W({\cal X}) $ solves the Lyapunov equation:
\begin{equation}
\label{WX_lyap}
\begin{bmatrix} A & 0 \\ 0 & {\cal XL}\end{bmatrix} \W({\cal X}) + \W({\cal X})\begin{bmatrix} A & 0 \\ 0 & {\cal XL}\end{bmatrix}^T+   {\cal B}({\cal X}) =0.
\end{equation}
\end{theorem}

\begin{proof}
To compute the gradient $\nabla f({\cal X})$, we write the derivative  $f'({\cal X}) \opd{\cal X}$ for some $\opd{\cal X} \in  \mathbb R^{(\nu+m) \times \nu}$ in gradient form using the trace. We introduce the gradient as:
\[ f'({\cal X})  \opd{\cal X} = \text{Trace} \left( \nabla f({\cal X})^T \opd{\cal X}  \right). \]
Then, we have:
\begin{equation*}
f'({\cal X})  \opd{\cal X}  =  \text{Trace} \left(     \M'({\cal X}) {\cal B}({\cal X})   + \M({\cal X}) {\cal B}'({\cal X})   \right).
\end{equation*}
We  compute separately the two terms in the above expression.  Let  $$\Phi({\cal X},\M)=\begin{bmatrix} A & 0 \\ 0 & {\cal XL} \end{bmatrix}^T \M  + \M  \begin{bmatrix} A & 0 \\ 0 & {\cal XL} \end{bmatrix}.$$ 
Since ${\cal X} \in {\cal D}_L$ and ${\cal D}_L$ is an open set, then by Lemma \ref{lema:lyap_sol}  we have that $ \Phi_{\M}({\cal X},\M) \opd \M$ given by:
\begin{align*}
 \Phi_{\M}({\cal X},\M) \opd \M &=\begin{bmatrix} A & 0 \\ 0 & {\cal XL} \end{bmatrix}^T \opd\M  + \opd\M  \begin{bmatrix} A & 0 \\ 0 & {\cal XL} \end{bmatrix}
\end{align*}
is surjective and also we have:
\begin{align*} 
\Phi_{ \cal X} ({\cal X},\M)  \opd{\cal X} &= \begin{bmatrix} 0 & 0 \\ 0 & \opd{\cal X} {\cal L} \end{bmatrix}^T \M + \M \begin{bmatrix} 0 & 0 \\\ 0 & \opd{\cal X} {\cal L} \end{bmatrix}.
\end{align*}
Since $\Phi({\cal X},\M) +{\cal C}=0$,   the Implicit Function Theorem yields the differentiability of $\M({\cal X})$ and the following relation:
\begin{align}
\label{eq_TFI1}
\begin{bmatrix} A & 0 \\ 0 & {\cal XL}\end{bmatrix}^T \M'({\cal X}) + \M'({\cal X}) \begin{bmatrix} A & 0 \\ 0 & {\cal XL}\end{bmatrix} & + \begin{bmatrix} 0 & 0 \\ 0 & \opd{\cal X} {\cal L} \end{bmatrix}^T \M({\cal X})   + \M({\cal X})\begin{bmatrix} 0 & 0 \\\ 0 & \opd{\cal X} {\cal L} \end{bmatrix}=0.
\end{align}

\noindent Moreover, by \eqref{W_Lyap} the  Gramian $\W({\cal X})$ is the unique solution of the Lyapunov equation \eqref{WX_lyap}. Subtracting \eqref{eq_TFI1} multiplied by $\W({\cal X})$ to the left from \eqref{WX_lyap} multiplied by $\M'({\cal X}) $ to the right,  taking the trace,  and reducing the appropriate terms, we get:
\begin{align}  
\label{eq:df1}
\text{Trace} \left( \M'({\cal X})   {\cal B}({\cal X})  \right) & =  \text{Trace} \left( \W({\cal X})\begin{bmatrix} 0 & 0 \\ 0 &  \opd{\cal X} {\cal L} \end{bmatrix}^T \M({\cal X}) + \M({\cal X}) \begin{bmatrix} 0 & 0 \\ 0 &  \opd{\cal X} {\cal L} \end{bmatrix} \W({\cal X})\right) \nonumber\\
& = 2 \text{Trace} \left(  {\cal L}  W_{12}^T ({\cal X})  M_{12}({\cal X}) \opd{\cal X} +  {\cal L}  W_{22} ({\cal X}) M_{22}({\cal X}) \opd{\cal X} \right), 
\end{align}
where for the second equality we used the block structure of $\W$ and $\M$ and the definition of trace.  Similarly, for the second term using the block structure of  $\M$ and the definition of trace, we get: 
\begin{align}  
\label{eq:df2}
\text{Trace} \left(  M({\cal X}) {\cal B}'({\cal X})   \right) & =  \text{Trace} \left( \M({\cal X})  \begin{bmatrix} 0 & B  {\cal E}^T  \opd{\cal X}^T   \\    \opd{\cal X} {\cal E} B^T &  \opd{\cal X}  {\cal E}  {\cal E}^T {\cal X}^T + {\cal X}  {\cal E}  {\cal E}^T  \opd{\cal X}^T\end{bmatrix}   \right) \nonumber \\
& = 2 \text{Trace} \left(  {\cal E} B^T  M_{12}({\cal X})  \opd{\cal X} +  {\cal E} {\cal E}^T X^T  M_{22}({\cal X}) \opd{\cal X} \right). 
\end{align}
Hence, from \eqref{eq:df1} and \eqref{eq:df2} we get the closed form expression for the gradient \eqref{eq_gradient1}.
\end{proof}

\noindent Note that the expression of the gradient   $\nabla f$ from \eqref{eq_gradient1} is the same as the partial gradient of the Lagrangian $ \nabla_{\cal X} \Gamma $ from \eqref{eq_KKT_cond_explicit}.  The result of previous theorem  also yields the necessary optimality condition for the model reduction Problem  \ref{prob_optH2_general}  expressed in terms of the optimization problem \eqref{eq_2norm_gram_err_general_m2}:
\begin{lemma}
If   ${\cal X} \in {\cal D}_L$ , where ${\cal X}=[S \; G]$, solves the optimization problem   \eqref{eq_2norm_gram_err_general_m2} corresponding to the model reduction Problem \ref{prob_optH2_general}, then 
\[   M_{12}^T({\cal X})  W_{12}({\cal X})  {\cal L}^T +    M_{22} ({\cal X}) W_{22}({\cal X}) {\cal L}^T +   M_{12}^T({\cal X}) B {\cal E}^T   +     M_{22}({\cal X}) {\cal X} {\cal E} {\cal E}^T =0,  \] 
where $\M({\cal X}) $ solves the Lyapunov equation \eqref{lyap_eq_12} and   $\W({\cal X}) $ solves the Lyapunov equation \eqref{WX_lyap}. 
\end{lemma}

\noindent We can replace the open set ${\cal D}_L$ with any sublevel set:
\[  {\cal N}_L^{{\cal X}_0} = \{ {\cal X} \in  {\cal D}_L: \; f({\cal X}) \leq f({\cal X}_0)  \},  \]
where ${\cal X}_0 \in  {\cal D}_L$ is any initial stable reduced order system matrix.  
 arguments as in \cite{Toi:85} we can show that ${\cal N}_L^{{\cal X}_0}$ is a compact set.   Then,  the theorem of Weierstrass implies that for any  given matrix ${\cal X}_0 \in  {\cal D}_L$ ,  the model reduction Problem  \ref{prob_optH2_general}  given by  optimization formulation \eqref{eq_2norm_gram_err_general_m2} has a global minimum  in the sublevel set ${\cal N}_L^{{\cal X}_0}$. We can also show that the gradient $\nabla f({\cal X})$ is Lipschitz continuous on the compact sublevel  set   ${\cal N}_L^{{\cal X}_0}$.  Let us briefly sketch the proof of this statement. First we observe that $\M({\cal X})$ and $\W({\cal X})$ are continuous functions and moreover there exists finite $\ell_M > 0$ such that:
\[   \|  \M({\cal X}) -   \M({\cal Y})  \| \leq \ell_M \|  {\cal X}  - {\cal Y}\| \quad \forall {\cal X},{\cal Y} \in   {\cal N}_L^{{\cal X}_0}.  \]
Then, using the expression of  $\nabla f({\cal X})$, compactness of  ${\cal N}_L^{{\cal X}_0}$, continuity of $\M({\cal X})$ and $\W({\cal X})$, and the previous relation we conclude that there exists $\ell_f >0$ such that:
\[  \|  \nabla f({\cal X}) -   \nabla f ({\cal Y})  \| \leq \ell_f  \|  {\cal X}  - {\cal Y}\|  \quad \forall {\cal X},{\cal Y} \in   {\cal N}_L^{{\cal X}_0}.  \]
This property of the gradient is useful when analyzing the convergence behavior of the first-order algorithm  we propose for solving \eqref{eq_2norm_gram_err_general_m2}.

%%%%%%%%%%%%%%%%%%%%%%%%%%%%%%%%%%%%%%%%%%%%%%

\subsubsection{SDP approach}
Alternatively, the non-convex problem \eqref{eq_2norm_gram_err_general_m} can be written equivalently in terms of matrix inequalities (semidefinite programming):
\begin{align}
\label{eq_2norm_gram_err_general_m22}
& \min_{(S, G, \M): \; \M \succeq 0} \text{Trace} \left( \begin{bmatrix} B \\ G \end{bmatrix}^T  \M \begin{bmatrix} B \\ G \end{bmatrix} \right) \\
& \text{s.t.}:  \begin{bmatrix} A & 0 \\ 0 & S-GL \end{bmatrix}^T \M  + \M  \begin{bmatrix} A & 0 \\ 0 & S-GL \end{bmatrix} + \begin{bmatrix} C^TC & -C^T (C\Pi) \\ -(C\Pi)^T C & (C\Pi)^T(C\Pi)  \end{bmatrix} \preceq 0, \nonumber
\end{align}
where recall that  $L$ is fixed a priori.  Clearly, SDP problem \eqref{eq_2norm_gram_err_general_m22} is not convex since it contains bilinear matrix inequalities (BMIs). However, next theorem proves that we can obtain a suboptimal solution through convex relaxation:   

\begin{theorem}
If the following convex SDP relaxation:
\begin{align}
\label{eq_2norm_gram_err_general_m5}
& \min_{(X_{22},Y_{22},Z_{22},\Theta_{22}), M_{11} \succeq 0, M_{22} \succeq 0} \text{Trace} \left( B^T M_{11}B + X_{22} \right) \\
& \text{s.t.}:  \Theta_{22}^T - L^T Z_{22}^T + \Theta_{22} - Z_{22}L +  C_V^TC_V \preceq Y_{22} \nonumber \\
& \qquad  \begin{bmatrix}  X_{22} & Z_{22} \\ Z_{22}^T & M_{22}\end{bmatrix} \succeq 0,\;  \begin{bmatrix} A^TM_{11} + M_{11}A^T +C^TC &  -C^TC_V \\
- C_V^T C & Y_{22} \end{bmatrix}  \preceq 0 \nonumber
\end{align}
has a solution, then we can recover a suboptimal solution of the model reduction Problem  \ref{prob_optH2_general}  expressed in terms of the SDP  problem \eqref{eq_2norm_gram_err_general_m22}  through the relations: 
\[ G= M_{22}^{-1} Z_{22}, \quad  S= M_{22}^{-1} \Theta_{22} \quad \text{and} \quad \M=\diag(M_{11}, M_{22}).   \] 
\end{theorem}

\begin{proof}
Using the block form of $\M$ and the notation $C_V=C\Pi = C V T$, \eqref{eq_2norm_gram_err_general_m22} yields the equivalent SDP problem:
\begin{align}
\label{eq_2norm_gram_err_general_m3}
& \min_{(S,G), \; \M \succeq 0} \text{Trace} \left( B^T M_{11}B + B^TM_{12}G + G^TM_{12}^T B + G^TM_{22}G \right) \\
& \text{s.t.}:
\begin{bmatrix} A^TM_{11} + M_{11}A^T +C^TC & A^T M_{12} +M_{12} (S-GL) -C^TC_V \\
M_{12}^TA + (S-GL)^TM_{12}^T  - C_V^T C & (S-GL)^T M_{22} + M_{22}(S-GL) +  C_V^TC_V \end{bmatrix}  \preceq 0. \nonumber
\end{align}

\noindent Note that problem \eqref{eq_2norm_gram_err_general_m3} is not convex since it contains bilinear matrix terms  and we are not aware of any change of variables that might lead to a convex reformulation (note that if we assume $M_{12} \not =0$, then we cannot convexify the previous BMIs since we need to define $M_{12} G= Z_{12}$ and $M_{22} G= Z_{22}$ and require $\M \succeq 0$).  However, if we assume the block $M_{12} =0$, then problem \eqref{eq_2norm_gram_err_general_m3} can be recast as a convex SDP. More precisely, if we introduce additional variables, then we have:
\begin{align}
\label{eq_2norm_gram_err_general_m4}
& \min_{(S,G,X_{22},Y_{22}), M_{11} \succeq 0, M_{22} \succeq 0} \text{Trace} \left( B^T M_{11}B + X_{22} \right) \\
& \text{s.t.}: X_{22} \succeq G^T M_{22} G, \;\; (S-GL)^T M_{22} + M_{22}(S-GL) +  C_V^TC_V \preceq Y_{22} \nonumber \\
&\qquad  \begin{bmatrix} A^TM_{11} + M_{11}A^T +C^TC &  -C^TC_V \\
- C_V^T C & Y_{22} \end{bmatrix}  \preceq 0.  \nonumber
\end{align}
Denoting $Z_{22} = M_{22}G, \Theta_{22} = M_{22}S$ and using the Schur complement,  problem \eqref{eq_2norm_gram_err_general_m4} becomes the convex SDP \eqref{eq_2norm_gram_err_general_m5}. Moreover, we can recover a suboptimal solution of the original problem through the relations:  $G= M_{22}^{-1} Z_{22}, S= M_{22}^{-1} \Theta_{22}$ and  $ \M=\text{diag}(M_{11}, M_{22})$.  Clearly, this is a suboptimal solution of the original SDP problem \eqref{eq_2norm_gram_err_general_m22}  since we restrict the matrix $\M$ to have the   block $M_{12} =0$. Hence, \eqref{eq_2norm_gram_err_general_m5} is  a convex SDP relaxation of the original problem \eqref{eq_2norm_gram_err_general_m22}.  
\end{proof}

%%%%%%%%%%%%%%%%%%%%%%%%%%%%%%%%%%%%%%%

\subsection{Numerical optimization algorithms for Problem 1}
\label{sec:numoptalg1}
In this section we present several optimization algorithms for solving the model reduction  Problem 1. For solving the associated KKT system \eqref{eq_KKT_cond_explicit} of  the non-convex problem \eqref{eq_2norm_gram_err_general_m} or the   non-convex (partial) optimization problem \eqref{eq_2norm_gram_err_general_m2} we propose first-order methods since they are adequate for  large-scale optimization problems, i.e. the dimension $n$ is very large.  Of course, we can also apply second-order methods to solve these optimization problems, but they require more expensive computations at each iteration (e.g., evaluation of Hessians and finding solutions of linear system), making them intractable when dimension $n$ of the original linear system \eqref{system} is  large.  

\subsubsection{Gradient type method for  KKT system}
One optimization algorithm that can be used for solving the KKT system \eqref{eq_KKT_cond_explicit} is the gradient method. Starting from an initial triplet $(\W_0,\M_0,{\cal X}_0)$ update:
\begin{align}
\label{eq_KKT_update}
&\W_{k+1}  =   \W_{k}  +  \alpha_k \nabla_{\W}  \Gamma (\W_k, \M_k, {\cal X}_k)\\
& \begin{bmatrix} \M_{k+1}  \\ {\cal X}_{k+1}  \end{bmatrix} = \begin{bmatrix} \M_{k} \\ {\cal X}_{k}    \end{bmatrix} - \alpha_k \nabla_{(\M, {\cal X})}  \Gamma (\W_k, \M_k, {\cal X}_k), \nonumber 
\end{align}
where $\alpha_k$ is a stepsize selected to minimize an appropriate merit function in the search direction at each step. Under some mild assumptions it is possible to show that the iterative process  \eqref{eq_KKT_update} converges locally to a KKT point, see, e.g., \cite{Lue:03}(Chapter 14). Moreover, if we start sufficiently close to a KKT point we can even choose $\alpha_k$ constant and the sequence will converge linearly to a KKT  point, with a speed of convergence depending on the starting point.   

\noindent If the  convex SDP relaxation \eqref{eq_2norm_gram_err_general_m5}  admits a solution, then we can consider as a starting point the suboptimal solution provided by this relaxation, i.e., ${\cal X}_0= [S_0 \ G_0] $ with $G_0= M_{22}^{-1} Z_{22}, S_0= M_{22}^{-1} \Theta_{22}$ and  $ \M_0=\text{diag}(M_{11}, M_{22})$. Moreover, we can take $\W_0$ as the solution   of the Lyapunov equation \eqref{W_Lyap} with  $S=S_0$ and $G=G_0$ given before.  Otherwise, we can fix $S_0$  and $L$ such that the pair $(L,S_0)$ is observable, and select a set $\{\lambda_1,\dots,\lambda_\nu\}\subset \mathbb{C}^-$. Then, from control theory it is known that there exists (stabilizing) $G_0$, computed by standard control algorithms, such that the spectrum $\sigma(S_0-G_0L)=\{\lambda_1,\dots, \lambda_\nu\}$. 

\noindent Based on the  explicit form of the KKT system   \eqref{eq_KKT_cond_explicit} we get the following simple iterative process:  
\begin{equation}
\label{eq_KKT_update1}
\begin{cases}
\W_{k+1} = \W_k+ \alpha_k ({\cal A}^T({\cal X}_k) \M_k + \M_k {\cal A}({\cal X}_k) + {\cal C}) \\
\M_{k+1} =  \M_k - \alpha_k ({\cal A}({\cal X}_k) \W_k + \W_k {\cal A}^T({\cal X}_k) + {\cal B}({\cal X}_k))\\
{\cal X}_{k+1}  =  {\cal X}_k - \alpha_k(M_{12,k}^T  B {\cal E}^T   +     M_{22,k} {\cal X}_k {\cal E} {\cal E}^T +  M_{12,k}^T W_{12,k}  {\cal L}^T +    M_{22,k}  W_{22,k} {\cal L}^T  ).
\end{cases}
\end{equation}
This algorithm has a cheap iteration since it  requires only matrix multiplications.  The update in \eqref{eq_KKT_update1} has the disadvantage however  that only the asymptotic ${\cal X}_k= [S_k \ G_k] $  leads to a reduced order stable  system while the intermediate iterates can lead to  unstable systems.

%%%%%%%%%%%%%%%%%%%%%%%%%%%%%%%%%%%

\subsubsection{Gradient method for partial minimization problem}
We have proved that the non-convex optimization problem  \eqref{eq_2norm_gram_err_general_m2} has differentiable objective function  and its gradient is given in \eqref{eq_gradient1}. Moreover, the gradient is Lipschitz continuous on any compact set. Then, we can apply gradient method for solving \eqref{eq_2norm_gram_err_general_m2}. Starting from  the initial stable matrix ${\cal X}_0 \in {\cal D}_L$ we consider the following update:
\[  {\cal X}_{k+1}  =  {\cal X}_k - \alpha_k \nabla f({\cal X}_k ), \]
where the stepsize $\alpha_k$ can be chosen by a backtracking procedure or constant in the interval  $(0, 2/\ell_f)$ (where $\ell_f$ denotes the Lpschitz constant of the gradient). With these choices for the stepsize and using the Lipschitz gradient property for the objective function the sequence of value functions $f({\cal X}_k)$ is nonincreasing \cite{Nes:04}:
\[  f({\cal X}_{k+1}) \leq f({\cal X}_k) - \Delta \cdot \| \nabla f({\cal X}_k ) \|^2 \quad \forall k \geq 0,  \]
for some constant $\Delta>0$.  Therefore all the iterates remain in the compact sublevel set ${\cal N}_L^{{\cal X}_0}$. Moreover, since $f$ is bounded from below by zero, then for any positive integer  $K$ it is straightforward to prove from the previous descent inequality   the following global convergence rate: 
\[   \min_{i=0:k}  \| \nabla f({\cal X}_k ) \|^2   \leq    \frac{f({\cal X}_0) - f^*}{\Delta \cdot k}  \quad \forall k \geq 0, \]
where $f^*$ is the optimal value of problem \eqref{eq_2norm_gram_err_general_m2}.  Under some mild assumptions, such as the Hessian of $f$ at a local minimum is positive definite and  bounded, then starting sufficiently close to this local optimum the gradient iteration converges linearly to this solution \cite{Nes:04}. Therefore, the speed of convergence of this iterative process depends on the starting point.  For choices of  the starting point we can consider the procedures described  in the previous section.    

\noindent Note that the gradient iteration has the  explicit form:
\begin{equation}
\label{eq_pm_update1}
{\cal X}_{k+1}  =  {\cal X}_k - \alpha_k(M_{12,k}^T  B {\cal E}^T   +     M_{22,k} {\cal X}_k {\cal E} {\cal E}^T +  M_{12,k}^T W_{12,k}  {\cal L}^T +    M_{22,k}  W_{22,k} {\cal L}^T ),
\end{equation}
where $\M_k$ and $\W_k$ are  the unique positive semidefinite  solutions of the Lyapunov equations in ${\cal X}_k$ \eqref{lyap_eq_12} and  \eqref{WX_lyap}, respectively. Therefore,  this iterative process has expensive iterations since it requires solving two Lyapunov equations, which can be  prohibitive when dimension $n$ of the original system is  large. On the other hand the update in \eqref{eq_pm_update1} has the advantage that any  iterate ${\cal X}_k= [S_k \ G_k]$ leads to a stable  reduced order model, while for the iteration  \eqref{eq_KKT_update1}  only the asymptotic ${\cal X}_k$  leads to a  stable  system.

%%%%%%%%%%%%%%%%%%%%%%%%%%%%%%%%%%%%%%%%%%%%

\subsubsection{Convex SDP relaxation}
There are several methods available for solving SDP problems with convex objective function and constraints of type BMIs, see, e.g., \cite{KocSti:12}. However, there are more efficient solvers  for  convex SDPs (as  problem \eqref{eq_2norm_gram_err_general_m5}) that can scale to large instances such as first order methods or interior point methods \cite{Nes:04}.   Note that in the general case,  i.e.,  for general matrices $A$,   the convex SDP relaxation \eqref{eq_2norm_gram_err_general_m5} is not exact, since  imposing the block $M_{12} =0$, its solution  is  suboptimal for the original SDP problem \eqref{eq_2norm_gram_err_general_m22}. If the  convex SDP relaxation \eqref{eq_2norm_gram_err_general_m5}  admits a solution, then we can initialize the gradient-based methods from previous two sections   with the suboptimal solution provided by this relaxation.  

On the other hand, for certain particular systems  the convex SDP relaxation \eqref{eq_2norm_gram_err_general_m5} is  exact. Indeed, this is the case, e.g., for \textit{positive systems}.   Let us briefly introduce  the notion of positive systems and their main properties, see, e.g., \cite{Ran:15} for a detailed exposition.  A matrix is said to be  \textit{Metzler} if all offdiagonal elements are non-negative. Further, the LTI system  \eqref{system} is  said to be a positive system if $A$ is Metzler and $B,  C \geq 0$. Then, one basic result for positive systems states that they  admit diagonal Lyapunov matrices:
\[  A \; \text{stable}   \; \iff \; \exists P \succ 0 \;\;  \text{diagonal \; s.t.}  \;\;  A^T P + P A \prec 0.  \] 
Recently,  a high interest in positive systems has been shown in the literature. Positive systems occur in modelling of applications with special structures from, e.g.,  biomedicine, economics, data networks, etc., \cite{Ran:15}.  Naturally, these systems are generally highly dimensional and  need to be approximated with the help of model order reduction techniques.  Unfortunately, conventional model reduction techniques do not preserve the positivity. However, working with an approximation violating basic physical constraints it always  leaves the question of how conclusive results on this basis are.  Recently, balanced truncation-based  methods  that preserve  positivity have been proposed in, e.g., \cite{LiLam:11,ReiVir:09}. Note that in all our optimization formulations  we proposed we can easily impose additional convex constraints for preserving positivity: $ \text{offdiagonal}(S-GL) \geq 0$  and $G \geq 0$, where $L$ is fixed a priori. Then, we can apply, e.g., a projected gradient type algorithm for solving the corresponding first two problems.  Moreover, for positive systems the convex  SDP relaxation \eqref{eq_2norm_gram_err_general_m5}  is exact since  there exists diagonal Gramian $\M$ satisfying the Lyapunov equation \eqref{M_Lyap} (see, e.g., \cite{LiLam:11}), and consequently requiring the block $M_{12} =0$ is not restricting the feasible set of the original SDP  problem   \eqref{eq_2norm_gram_err_general_m22}. Furthermore,   in the  convex SDP problem \eqref{eq_2norm_gram_err_general_m5} positivity   can be imposed through new additional convex constraints:
\[    \text{offdiagonal}(\Theta_{22} - Z_{22} L) \geq 0, \;  Z_{22} \geq 0.    \]
It is clear that the reduced order model is also a positive system, i.e.,  $ \text{offdiagonal}(S-GL) \geq 0$  and $G \geq 0$,  provided that $G= M_{22}^{-1} Z_{22}, S= M_{22}^{-1} \Theta_{22}$, $\M$ diagonal, and $\Theta_{22}, Z_{22}$ satisfy the new constraints from above.    Hence, our model reduction techniques are flexible, allowing to incorporate easily constraints for preserving positivity and/or stability.

%%%%%%%%%%%%%%%%%%%%%%%%%%%%%%%%%%%%%%%%%%5

\subsection{Optimization formulation of Problem \ref{prob_optH2_G}}
Using similar arguments as for  the general Problem \ref{prob_optH2_general} we can derive optimization formulations for the particular  Problem \ref{prob_optH2_G}, where now the parametrisation of the reduced order model is done only through matrix $G$. Note that in  Problem \ref{prob_optH2_G} the pair  $(S,L)$ is fixed a priori  such that it is observable. Further,  we can find $C \Pi$ based on  Corollary \ref{cor_momPiY_Krylov}. Moreover, if we choose $S$ unstable, that is $\sigma(S) \subseteq  \mathbb C^+$, then the optimal solution of  Problem \ref{prob_optH2_G} automatically satisfies  $\sigma(S) \cap \sigma(S-GL) = \emptyset$. Then, from \eqref{eq_2norm_gram} it follows that  Problem \ref{prob_optH2_G} can be written as: 
\begin{align}
\label{eq_2norm_gram_err}
\min_{G\ \text{s.t.}\ \sigma(S - GL)\subset\mathbb C^-}\|\K\|_2^2 &=\min_{(G,\W) \ \text{s.t.}\ \sigma(S - GL)\subset\mathbb C^-, \ \eqref{W_Lyap}} C\begin{bmatrix} I &  -\Pi\end{bmatrix}\begin{bmatrix} W_{11} & W_{12} \\ W_{12}^T & W_{22}\end{bmatrix}\begin{bmatrix} I \\  -\Pi^T\end{bmatrix}C^T \nonumber\\
&= \min_{(G,\M) \ \text{s.t.} \ \sigma(S - GL)\subset\mathbb C^-, \ \eqref{M_Lyap}} \begin{bmatrix} B^T & G^T \end{bmatrix} \begin{bmatrix} M_{11} & M_{12} \\ M_{12}^T & M_{22}\end{bmatrix} \begin{bmatrix} B \\ G \end{bmatrix}.\nonumber
\end{align}

\noindent Below we consider again only the formulation in terms of the observability Gramian $\M$:
\begin{align}
\label{eq_2norm_gram_err_G_m1}
& \min_{(G,\M)} \text{Trace}(\B_e^T \M \B_e) \\
& \text{s.t.}: \;\; \sigma(S-GL) \subset \mathbb C^- \quad \text{and} \quad \A_e^T \M + \M \A_e + \C_e^T \C_e =0, \nonumber
\end{align}
with $(S,L)$ fixed and  $\Pi=VT$, where $V$ as in  \eqref{projection_V} and $T$ some fixed non-singular matrix. We clearly observe that in this case the reduced order model is parametrized only in the matrix $G$. Let us denote:
\[ {\cal A}(G) = \begin{bmatrix} A & 0 \\ 0 & S-GL \end{bmatrix}, \;\;   {\cal B}(G) =  \begin{bmatrix} B \\ G \end{bmatrix} \begin{bmatrix} B \\ G  \end{bmatrix}^T, \;\; {\cal C} = \C_e^T \C_e =  \begin{bmatrix} I \\  -T^TV^T\end{bmatrix}C^TC\begin{bmatrix} I \\  -T^TV^T\end{bmatrix}^T.  \]
In the next sections we present several (equivalent) reformulations of the nonconvex problem \eqref{eq_2norm_gram_err_G_m1}, accompanied by their first-order  optimality conditions.% Although, in the derivations we consider the same tools we used for the general Problem 1, for completeness we present the details below. 

%%%%%%%%%%%%%%%%%%%%%%%%%%%%%%%%

\subsubsection{ KKT approach}
We determine the corresponding  KKT system for  optimization problem  \eqref{eq_2norm_gram_err_G_m1}. We first define  the open set  ${\cal D}_{(SL)} =\{  G:  \;  \sigma(S-GL)  \subset \mathbb C^- \}$ where the pair $(S,L)$ is fixed a priori.  Using again that $\text{Trace}(MN) = \text{Trace}(NM)$,    Lagrangian function associated to problem  \eqref{eq_2norm_gram_err_G_m1}  is given by:
\begin{equation}
\label{eq_Lagrange_KKT_G}
\Gamma(\W, \M, G) = \text{Trace}( \M {\cal B}(G) ) + \text{Trace} (\W ({\cal  A}^T(G) \M + \M {\cal A}(G) + {\cal C} )),
\end{equation}
where the multiplier $\W$ is associated to the equality constraint in  \eqref{eq_2norm_gram_err_G_m1}.  Then, we write \eqref{eq_2norm_gram_err_G_m1} into the max-min form:
\begin{equation}
\label{eq_2norm_gram_err_G_minmax_KKT}
\max_{\W}  \min_{\M, G \in {\cal D}_{(SL)}}  \Gamma(\W, \M, G).
\end{equation}
From standard optimization arguments we know that any solution   $(\W, \M, G)$ of problem \eqref{eq_2norm_gram_err_G_minmax_KKT} implies that  $(\M, G)$ is a (possibly local)  minimum point of the original problem \eqref{eq_2norm_gram_err_G_m1} and needs to  satisfy the KKT system:
\[ \nabla  \Gamma (\W, \M, G) =0 \quad \iff \quad
\begin{cases}
\nabla_{\W}  \Gamma (\W, \M,G) =0\\
\nabla_{(\M,G)}  \Gamma (\W, \M,G) =0.
\end{cases}
\]

\noindent Next theorem derives explicitly the corresponding \emph{KKT system}:
\begin{theorem}
The KKT system of optimization problem  \eqref{eq_2norm_gram_err_G_m1} is given by:
\begin{equation}
\label{eq_KKT_cond_explicit_G}
\nabla  \Gamma (\W, \M,G) = 0 \; \iff \;
\begin{cases}
{\cal A}^T(G) \M +  \M  {\cal A}(G)  + {\cal C}  = 0\\
{\cal A}(G) \W +  \W  {\cal A}^T(G) + {\cal B} (G) =0 \\
M_{12}^T B    +     M_{22} G -  M_{12}^T W_{12}  {L}^T -   M_{22}  W_{22} {L}^T  = 0.
\end{cases}
\end{equation}
\end{theorem}

\begin{proof}
Note that the KKT system has the following explicit form:
\[
\nabla  \Gamma (\W, \M, G) =
\begin{bmatrix}
\nabla_\W \Gamma (\W, \M, G) \\
\nabla_\M \Gamma (\W, \M, G)  \\
\nabla_{G} \Gamma (\W, \M, G)
\end{bmatrix}
=
\begin{bmatrix}
{\cal A}^T(G) \M +  \M  {\cal A}(G)  + {\cal C}  \\
{\cal A}(G) \W +  \W  {\cal A}^T(G) + {\cal B} (G) \\
\nabla_{G} \Gamma (\W, \M, G)
\end{bmatrix}.
\]
It remains to explicitly compute $ \nabla_{G} \Gamma (\W, \M, G)$.  However, to write  the gradient expression, we  introduce a gradient of $\Gamma$ w.r.t. $G$ using the trace of a matrix:
\[    \Gamma' (\W, \M, G) \opd{G}  =    \text{Trace}( \nabla_{G}^T  \Gamma (\W, \M, G) \opd{G}), \quad \text{with} \; \opd{G} \in  \mathbb R^{\nu \times m}. \]
Then:
\begin{align*}
&\text{Trace}( \nabla_{G}^T \Gamma (\W, \M, G) \opd{G}) \\
& =  \text{Trace}( \M {\cal B}' ({G})) + \text{Trace} (\W ( ({\cal A}'({G}))^T \M + \M {\cal  A}'({G})))\\
&=  \text{Trace} \left( \M  \begin{bmatrix} 0 & B  \opd{G}^T \\  \opd{G}  B^T &   \opd{G}   {G}^T   +  {G}   \opd{G}^T \end{bmatrix} + \W \begin{bmatrix} 0 & 0 \\\ 0 & -\opd{G} {L} \end{bmatrix}^T \M + \M \begin{bmatrix} 0 & 0 \\\ 0 & -\opd{G} {L} \end{bmatrix}  \W  \right)\\
&= 2 \text{Trace} \left(   B^T M_{12}  \opd{G} +  G^T M_{22} \opd{G}  - {L} W_{12}^T M_{12} \opd{G}   -  {L} W_{22} M_{22} \opd{G}  \right),
\end{align*}
where in the last equality we used the block structure of $\W$ and $\M$.  Then:
\[  \nabla_{G} \Gamma (\W, \M, {G}) =  2  \left( M_{12}^T B    +     M_{22} {G}  -  M_{12}^T W_{12}  {L}^T  -    M_{22}  W_{22} {L}^T  \right).   \]
Hence, we get the KKT system from \eqref{eq_KKT_cond_explicit_G}.
\end{proof}

\noindent The result of previous theorem  also yields the necessary optimality condition for the optimization problem   \eqref{eq_2norm_gram_err_G_m1} of the model reduction Problem  \ref{prob_optH2_G}:
\begin{lemma}
If   $\M $ and ${G} \in {\cal D}_{(SL)}$  solves the optimization problem   \eqref{eq_2norm_gram_err_G_m1} corresponding to the model reduction Problem \ref{prob_optH2_G}, then there exists $\W$ such that the triplet $(\W,\M,{G})$ solves the KKT system \eqref{eq_KKT_cond_explicit_G}.
\end{lemma}

%%%%%%%%%%%%%%%%%%%%%%%%%%%%%%%%%%%%%5

\subsubsection{Partial minimization approach}
\label{sect_gradientH2_G}
Consider now the optimization problem  \eqref{eq_2norm_gram_err_G_m1} where recall that the pair $(S,L)$ is fixed a priori, $\Pi= VT$, and $\A$ depends on $G$, i.e.,  $\A = {\cal A}(G)$. Then, the partial minimization holds for \eqref{eq_2norm_gram_err_G_m1}:  
\[ \eqref{eq_2norm_gram_err_G_m1}  = \min_{G: \sigma(S-GL) \subset \mathbb C^-}  \left(  \min_{\M: \A_e^T \M + \M \A_e + \C_e^T \C_e =0} \text{Trace}(\B_e^T \M \B_e)) \right).  \]
However, if $S-GL$ and  $A$ are stable,  it follows from Lemma \ref{lema:lyap_sol}  that there  exists unique  $\M =\M(G) \succeq 0$  solution of the Lyapunov equation:
\[  \A_e^T \M + \M \A_e + \C_e^T \C_e =0.  \]
Hence, for any  stabilizable $G$, the partial minimization in $\M$ leads to an optimal value: 
\[  f(G) = \min_{\M: \A_e^T \M + \M \A_e + \C_e^T \C_e =0} \text{Trace}(\B_e^T \M \B_e)) = \text{Trace} \left( \begin{bmatrix} B \\ G \end{bmatrix}^T  \M(G) \begin{bmatrix} B \\ G \end{bmatrix} \right),  \]
where $\M(G)$ is the unique solution of the Lyapunov equation:
\begin{align}
\label{lyap_eq_11_G}
\begin{bmatrix} A & 0 \\ 0 & S-GL \end{bmatrix}^T \M  + \M  \begin{bmatrix} A & 0 \\ 0 & S-GL \end{bmatrix} + \begin{bmatrix} C^TC & -C^T C_V \\ -C_V^T C & C_V^T C_V  \end{bmatrix} =0,
\end{align}
with $C_V = C \Pi = C V T$. Explicitly, in terms of $G$, we have:
\begin{align}
\label{eq_2norm_gram_err_G_m2}
& \min_{G}  f(G) \;\; \left( =\text{Trace} \left( \begin{bmatrix} B \\ G \end{bmatrix}^T  \M(G) \begin{bmatrix} B \\ G \end{bmatrix} \right)\right) \\
& \text{s.t.}: \;\; \sigma(S-GL) \subset \mathbb C^- \quad \text{and} \quad \eqref{lyap_eq_G_1}, \nonumber
\end{align}
where $\M(G)$ is the solution of the Lyapunov equation:
\begin{align}
\label{lyap_eq_G_1}
\begin{bmatrix} A & 0 \\ 0 & S-GL \end{bmatrix}^T \M(G)  + \M(G)  \begin{bmatrix} A & 0 \\ 0 & S-GL \end{bmatrix} + \begin{bmatrix} C^TC & -C^T C_V \\ -C_V^T C & C_V^T C_V  \end{bmatrix} =0.
\end{align}

\noindent We now compute the expression of the gradient of the objective function of \eqref{eq_2norm_gram_err_G_m2}.  Using again that $\text{Trace}(MN) = \text{Trace}(NM)$, the non-convex objective function of \eqref{eq_2norm_gram_err_G_m2} becomes:
\[  f(G) =  \text{Trace} \left( \begin{bmatrix} B \\ G \end{bmatrix}^T  \M(G) \begin{bmatrix} B \\ G \end{bmatrix} \right) = \text{Trace} \left(  \M(G) \begin{bmatrix} B \\ G \end{bmatrix} \begin{bmatrix} B \\ G \end{bmatrix}^T  \right) = \text{Trace} \left(  \M(G)  {\cal B}(G) \right). \]

\begin{theorem}
The  objective function  $f$ of \eqref{eq_2norm_gram_err_G_m2}  is differentiable on the set of stable matrices ${\cal D}_{(SL)}$  and the gradient of $f$ at $G \in {\cal D}_{(SL)}$ is given by:
\begin{align}
\label{eq_gradient1_G}
\nabla f(G) =  2  \left[- M_{12}^T(G)  W_{12}(G)  {L}^T -    M_{22}(G) W_{22}(G) {L}^T +   M_{12}^T(G) B    +     M_{22}(G) G  \right],
\end{align}
where $\M(G) $ solves the Lyapunov equation \eqref{lyap_eq_G_1} and   $\W(G) $ solves the Lyapunov equation:
\begin{equation}
\label{WX_lyap_G}
\begin{bmatrix} A & 0 \\ 0 & S-GL \end{bmatrix} \W(G) + \W(G) \begin{bmatrix} A & 0 \\ 0 & S- GL \end{bmatrix}^T+   {\cal B}(G) =0.
\end{equation}
\end{theorem}

\begin{proof}
To  computing the expression of the gradient $\nabla f(G) \in \rset^{\nu \times m}$ we write the derivative  $f'(G)  \opd{G}$ for some $\opd{G} \in \rset^{\nu \times m}$ in gradient form using the trace. We introduce the gradient as:
\[ f'(G)  \opd{G} = \text{Trace} \left( \nabla f(G)^T \opd{G}  \right). \]
From the expression of $f(G)$ we have:
\[ f'(G)  \opd{G} =  \text{Trace} \left(  \M'(G) {\cal B}(G) + \M(G) {\cal B}'(G) \right). \]
We compute separately the two terms in the above expression.  Let 
\[  \Phi(G,\M) = \begin{bmatrix} A & 0 \\ 0 & S-GL \end{bmatrix}^T \M  + \M  \begin{bmatrix} A & 0 \\ 0 & S-GL \end{bmatrix}.  \]
Since  $G \in {\cal D}_{(SL)}$ and  ${\cal D}_{(SL)}$ is an open set, then by Lemma \ref{lema:lyap_sol}  we have that $ \Phi_{\M}({G},\M) \opd{\M}$ given by:
\begin{align*}
\Phi_{\M}({G},\M) \opd{\M} &=\begin{bmatrix} A & 0 \\ 0 & S-GL \end{bmatrix}^T \opd{\M}  + \opd{\M}  \begin{bmatrix} A & 0 \\ 0 & S-GL \end{bmatrix}
\end{align*}
is surjective and also we have:
\begin{align*}
\Phi_{G}({G},\M)  \opd{G} &= \begin{bmatrix} 0 & 0 \\ 0 & -\opd{G} {L} \end{bmatrix}^T \M + \M \begin{bmatrix} 0 & 0 \\  0 & -\opd{G} {L} \end{bmatrix}.
\end{align*}
Since $\Phi(G,\M) +{\cal C}=0$,   the Implicit Function Theorem yields the differentiability of $\M(G)$ and the following relation:
\begin{align}
\label{eq_TFI1_G}
\begin{bmatrix} A & 0 \\ 0 & S-GL \end{bmatrix}^T \M'(G)  + \M'(G) \begin{bmatrix} A & 0 \\ 0 & S-GL \end{bmatrix} &  + \begin{bmatrix} 0 & 0 \\ 0 & - \opd{G} {L} \end{bmatrix}^T \M(G) \nonumber \\
&   + \M(G) \begin{bmatrix} 0 & 0 \\\ 0 & -\opd{G} {L} \end{bmatrix}=0.
\end{align}
Consider the Lyapunov equation \eqref{W_Lyap} with the unique  solution $\W(G)$ provided $G$ is stabilizable:
\begin{align}
\label{lyap2_G}
\begin{bmatrix} A & 0 \\ 0 & S-GL \end{bmatrix} \W(G)  + \W(G)   \begin{bmatrix} A & 0 \\ 0 & S-GL \end{bmatrix}^T + {\cal B}(G) =0.
\end{align}
Subtracting \eqref{eq_TFI1_G} multiplied by $\W(G)$ to the left from \eqref{lyap2_G} multiplied by $\M'(G)$ to the right, taking the trace, and reducing the appropriate terms,  yields:
\begin{align}  
\label{eq:df1_G}
\text{Trace} \left( \M'(G) {\cal B}(G)  \right) & =  \text{Trace} \left( \W(G)\begin{bmatrix} 0 & 0 \\ 0 & -{\opd G} L \end{bmatrix}^T \M(G) + \M(G) \begin{bmatrix} 0 & 0 \\ 0 & -{\opd G} L \end{bmatrix} \W(G)\right) \nonumber  \\
&=  2 \text{Trace} \left(- L W_{12}^T(G)  M_{12}(G)  {\opd G}  -  L   W_{22}(G) M_{22}(G)  {\opd G} \right).
\end{align}

\noindent Similarly, for the second term using the block structure of  $\M$ and the definition of trace, we get:
\begin{align}
\label{eq:df2_G}
\text{Trace} \left(  M(G) {\cal B}'({G})   \right) & =  \text{Trace} \left( \M({G})  \begin{bmatrix} 0 & B   \opd{G}^T   \\    \opd{G} B^T &  \opd{G} {G}^T + {G}  \opd{G}^T \end{bmatrix}   \right) \nonumber \\
& = 2 \text{Trace} \left(  B^T  M_{12}(G)  \opd{G} +  G^T  M_{22}({G}) \opd{G} \right).
\end{align}
Hence, from \eqref{eq:df1_G} and \eqref{eq:df2_G} we get the closed form expression for the gradient \eqref{eq_gradient1_G}.
\end{proof}

\noindent Note that the expression of the gradient   $\nabla f$ from \eqref{eq_gradient1_G} is the same as the partial gradient of the Lagrangian $ \nabla_{G} \Gamma $ from \eqref{eq_KKT_cond_explicit_G}.  The previous result also yields the necessary optimality condition for the model reduction Problem  \ref{prob_optH2_G}  expressed in terms of the optimization problem \eqref{eq_2norm_gram_err_G_m2}:
\begin{lemma}
If   ${G} \in {\cal D}_{(SL)}$  solves the optimization problem   \eqref{eq_2norm_gram_err_G_m2} corresponding to the model reduction Problem \ref{prob_optH2_G}, then
\[   M_{12}^T({G})  W_{12}({G})  {L}^T +   M_{22} ({G}) W_{22}({G}) {L}^T =    M_{12}^T({G}) B    +     M_{22}({G}) {G}   \]
where $\M({G}) $ solves the Lyapunov equation \eqref{lyap_eq_G_1} and   $\W({G}) $ solves the Lyapunov equation \eqref{WX_lyap_G}.
\end{lemma}

\noindent We can replace the open set ${\cal D}_{(SL)}$ with any sublevel set:
\[  {\cal N}_{(SL)}^{{G}_0} = \{ G  \in  {\cal D}_{(SL)}: \; f({G}) \leq f({G}_0)  \},  \]
where ${G}_0 \in  {\cal D}_{(SL)}$ is any initial stable reduced order system matrix.  Using similar arguments as in \cite{Toi:85} we can show that ${\cal N}_{(SL)}^{{G}_0}$ is a compact set.   Then,  the theorem of Weierstrass implies that for any  given matrix ${G}_0 \in  {\cal D}_{(SL)}$ ,  the model reduction Problem  \ref{prob_optH2_G}  given by  optimization formulation \eqref{eq_2norm_gram_err_G_m1} has a global minimum  in the sublevel set ${\cal N}_{(SL)}^{{G}_0}$. We can also show that the gradient $\nabla f({G})$ is Lipschitz continuous on the compact sublevel  set   ${\cal N}_{(SL)}^{{G}_0}$.  Let us briefly sketch the proof of this statement. First we observe that $\M({G})$ and $\W({G})$ are continuous functions and moreover there exists finite $\ell_M > 0$ such that:
\[   \|  \M({G}) -   \M(\bar{G})  \| \leq \ell_M \|  {G}  - \bar{G} \| \quad \forall {G}, \bar{G} \in   {\cal N}_{(SL)}^{{G}_0}.  \]
Then, using the expression of  $\nabla f({G})$, compactness of  ${\cal N}_{(SL)}^{{G}_0}$, continuity of $\M({G})$ and $\W({G})$, and the previous relation we conclude that there exists $\ell_f >0$ such that:
\[  \|  \nabla f({G}) -   \nabla f (\bar{G})  \| \leq \ell_f  \|  {G}  - \bar{G}\|  \quad \forall {G}, \bar{G} \in   {\cal N}_{(SL)}^{{G}_0}.  \]
This property of the gradient is useful when analyzing the convergence behavior of the algorithm  we propose for solving the optimization problem \eqref{eq_2norm_gram_err_G_m2}.

%%%%%%%%%%%%%%%%%%%%%%%%%%%%%%%%%%%%%%%%%%%%

\subsubsection{SDP approach}
Alternatively, problem \eqref{eq_2norm_gram_err_G_m1} can be written equivalently in terms of matrix inequalities (semidefinite programming):
\begin{align}
\label{eq_2norm_gram_err_G_m3}
& \min_{(G, \M)} \text{Trace} \left( \begin{bmatrix} B \\ G \end{bmatrix}^T  \M \begin{bmatrix} B \\ G \end{bmatrix} \right) \\
& \text{s.t.}: \;\; \M \succeq 0 \quad \text{and} \quad
\begin{bmatrix} A & 0 \\ 0 & S-GL \end{bmatrix}^T \M  + \M  \begin{bmatrix} A & 0 \\ 0 & S-GL \end{bmatrix} + \begin{bmatrix} C^TC & -C^T C_V \\ -C_V^T C & C_V^TC_V  \end{bmatrix} \preceq 0, \nonumber
\end{align}
where recall that the pair $(S, L)$ is fixed a priori.  Clearly, SDP problem \eqref{eq_2norm_gram_err_G_m3} is not convex since it contains bilinear matrix inequalities. However, next theorem proves that we can obtain a suboptimal solution through convex relaxation.

\begin{theorem}
If the following convex SDP relaxation:
\begin{align}
\label{eq_2norm_gram_err_G_m6}
& \min_{(G,X_{22},Y_{22},Z_{22}), \; M_{11} \succeq 0, M_{22} \succeq 0} \text{\emph{Trace}} \left( B^T M_{11}B + X_{22} \right) \\
& \text{s.t.}:  S^T M_{22} - L^T Z_{22}^T + M_{22}S - Z_{22}L +  C_V^TC_V \preceq Y_{22} \nonumber \\
& \qquad  \begin{bmatrix}  X_{22} & Z_{22} \\ Z_{22}^T & M_{22}\end{bmatrix} \succeq 0,\;  \begin{bmatrix} A^TM_{11} + M_{11}A^T +C^TC &  -C^TC_V \\
- C_V^T C & Y_{22} \end{bmatrix}  \preceq 0  \nonumber
\end{align}
has a solution, then we can recover a suboptimal solution of the model reduction Problem  \ref{prob_optH2_G}  expreesed in terms of the SDP  problem \eqref{eq_2norm_gram_err_G_m1}  through the relations:
\[ G= M_{22}^{-1} Z_{22} \quad \text{and} \quad \M=\text{\emph{diag}}(M_{11}, M_{22}).   \]
\end{theorem}

\begin{proof}
Using the block form of $\M$ we get the equivalent problem:
\begin{align}
\label{eq_2norm_gram_err_G_m4}
& \min_{G, \; \M \succeq 0} \text{Trace} \left( B^T M_{11}B + B^TM_{12}G + G^TM_{12}^T B + G^TM_{22}G \right) \\
& \text{s.t.}:
\begin{bmatrix} A^TM_{11} + M_{11}A^T +C^TC & A^T M_{12} +M_{12} (S-GL) -C^TC_V \\
(S-GL)^TM_{12}^T + M_{12}^TA - C_V^T C & (S-GL)^T M_{22} + M_{22}(S-GL) +  C_V^TC_V \end{bmatrix}  \preceq 0. \nonumber
\end{align}
If we introduce additional variables we can reformulate the previous problem as an SDP subject to bilinear matrix inequalities. Indeed, we have the equivalent formulation:
\begin{align}
\label{eq_2norm_gram_err_G_m44}
& \min_{(G,X_{22},Y_{22}),\;  \M \succeq 0  } \text{Trace} \left( B^T M_{11}B + X_{22} \right) \\
& \text{s.t.}: X_{22} \succeq B^TM_{12}G + G^TM_{12}^T B + G^T M_{22} G \nonumber\\ & \qquad (S-GL)^T M_{22} + M_{22}(S-GL) + C_V^TC_V \preceq Y_{22} \nonumber \\
&\qquad  \begin{bmatrix} A^TM_{11} + M_{11}A^T +C^TC &  A^T M_{12} +M_{12} (S-GL)-C^TC_V \nonumber \\
(S-GL)^TM_{12}^T + M_{12}^TA - C_V^T C & Y_{22} \end{bmatrix}  \preceq 0, \nonumber
\end{align}
and  using now Schur complement we arrive at an SDP with convex objective function but with non-convex constraints of type BMIs:
\begin{align}
\label{eq_2norm_gram_err_G_m444}
& \min_{(G,X_{22},Y_{22}),\; \M \succeq 0} \text{Trace} \left( B^T M_{11}B + X_{22} \right) \\
& \text{s.t.}: (S-GL)^T M_{22} + M_{22}(S-GL) + C_V^TC_V \preceq Y_{22} \nonumber \\
&\qquad  \begin{bmatrix} X_{22} - B^TM_{12}G - G^TM_{12}^T B &  G^T M_{22} \\ M_{22} G  & M_{22} \end{bmatrix}  \succeq 0 \nonumber\\
&\qquad  \begin{bmatrix} A^TM_{11} + M_{11}A^T +C^TC &  A^T M_{12} +M_{12} (S-GL)-C^TC_V \\ (S-GL)^TM_{12}^T + M_{12}^TA - C_V^T C & Y_{22} \end{bmatrix}  \preceq 0, \nonumber
\end{align}
Problem \eqref{eq_2norm_gram_err_G_m444} is not convex since it contains bilinear matrix terms. However, if we assume that the block $M_{12} =0$, then problem \eqref{eq_2norm_gram_err_G_m444} can be recast as a  convex SDP. That is, for $M_{12} =0$ from \eqref{eq_2norm_gram_err_G_m444} we get:
\begin{align}
\label{eq_2norm_gram_err_G_m5}
& \min_{(G,X_{22},Y_{22}),\;  M_{11} \succeq 0, M_{22} \succeq 0} \text{Trace} \left( B^T M_{11}B + X_{22} \right) \\
& \text{s.t.}: X_{22} \succeq G^T M_{22} G, \;\; (S-GL)^T M_{22} + M_{22}(S-GL) + C_V^TC_V \preceq Y_{22} \nonumber \\
&\qquad  \begin{bmatrix} A^TM_{11} + M_{11}A^T +C^TC &  -C^TC_V \\
- C_V^T C & Y_{22} \end{bmatrix}  \preceq 0. \nonumber
\end{align}
Denoting by $Z_{22} = M_{22}G$ and using the Schur complement, problem \eqref{eq_2norm_gram_err_G_m5} becomes the convex SDP \eqref{eq_2norm_gram_err_G_m6}.  Moreover, if \eqref{eq_2norm_gram_err_G_m6} has a solution, then we can recover $G= M_{22}^{-1} Z_{22}$ and $\M = \text{diag}(M_{11},M_{22})$.   Note also that the solution $(G,\M)$ of this convex SDP problem is a suboptimal solution of the original problem \eqref{eq_2norm_gram_err_G_m1} since we restrict $M_{12} =0$. 
\end{proof}

\noindent All the numerical optimization algorithms presented in Section \ref{sec:numoptalg1} can be applied to also solve the three optimization problems  from this section corresponding to the relaxed Problem 2.  Thus we  omit these details here and refer to Section \ref{sec:numoptalg1}.

%%%%%%%%%%%%%%%%%%%%%%%%%%%%%%%%%%%%%%%%%

\section{Illustrative examples}
\label{sect_expl}
\noindent In this section, we illustrate the efficiency of our results numerically on examples such as a double-pendulum \cite{fujimoto-ono-hayakawa-CDC2010} or a CD player \cite{gugercin-antoulas-beattie-SIAM2008}. %, and transportation networks modeled as positive systems \cite{Ran:15}. 
In particular, we compute and compare  reduced order models   for these test systems  achieving (possibly) the minimum  $H_2$-norm.

\subsection{Cart with a double pendulum controller}
\noindent Consider the following cart system with a double-pendulum controller,
depicted in Figure \ref{fig_cart_pendulum}, see also \cite{fujimoto-ono-hayakawa-CDC2010,ionescu-astolfi-CDC2013}. 
\begin{figure}[h]
\centering\small \psfrag{K}{$k$}\psfrag{U1}{$\mu_{1}$} \psfrag{Q1}{$q_{1}$}
\psfrag{M1}{$m_{1}$}\psfrag{U2}{$\mu_{2}$}\psfrag{Q2}{$q_{2}$}\psfrag{L2}{$l_{2}$}\psfrag{M2}{$m_{2}$}
\psfrag{U3}{$\mu_{3}$}\psfrag{Q3}{$q_{3}$}\psfrag{L3}{$l_{3}$}\psfrag{M3}{$m_{3}$}
\includegraphics[clip,scale=0.75]{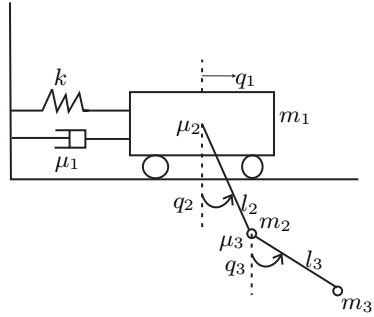}
\caption{Cart system with a double-pendulum controller}
\label{fig_cart_pendulum}
\end{figure}

\noindent Defining the state as $x=[q_{1}\ \dot{q}_{1\,}q_{2}\ \dot{q}_{2}\ q_{3}\ \dot{q}_{3}]^{T}\in\mathbb{R}^{6}$
and selecting the output as $y=x_{1}$, we obtain a 6th order system described
by equations of the form (\ref{system}), with:
\begin{align}
\label{eq_CL_example}
A&=\begin{bmatrix}0 & 1 & 0 & 0 & 0 & 0\\
-1 & -1 & 98/5 & 1 & 0 & 0\\
0 & 0 & 0 & 1 & 0 & 0\\
1 & 1 & -196/5 & -2 & 49/5 & 1\\
0 & 0 & 0 & 0 & 0 & 1\\
0 & 0 & 98/5 & 1 & -98/5 & -2
\end{bmatrix},\,B=\begin{bmatrix}0\\
1\\
0\\
-1\\
0\\
0
\end{bmatrix},\,  C=\begin{bmatrix}1&
0&
0&
0&
0&
0
\end{bmatrix}.
\end{align}

\noindent Fix the observable pair $(S,L)$: 
$$ S=
\begin{bmatrix}0 & 1\\
0 & 0
\end{bmatrix}\quad \text{and}\quad L=[1\,\,1].
$$ 
A family of second order models
that match the first two moments at zero of (\ref{eq_CL_example})
is described by $\Sigma_{G}$ as in (\ref{redmod_CPi}) with
\begin{equation}\label{cpend_redmod}
F=\begin{bmatrix}-g_{1} & 1\\
-g_{2} & 0
\end{bmatrix},\,G=\begin{bmatrix}g_{1}\\
g_{2}
\end{bmatrix},\,H=[1\,\,-1],
\end{equation}
where $g_{1},\,g_{2}\in\mathbb{C}$ are the  free parameters.  Selecting e.g. $g_1=1$ and $g_2=0.5$ yields a model \eqref{cpend_redmod} with $$F=\begin{bmatrix}-1 & 1\\
-0.5 & 0
\end{bmatrix},\,G=\begin{bmatrix}1\\
0.5
\end{bmatrix},\,H=[1\,\,-1],$$ which is stable, but not a very accurate approximation error as revealed by the first row in Table I.  By \cite{i-TAC2016}, the unique second order
model that matches the first two moments of $K(s)$ and $K'(s)$ is
\[
F=\begin{bmatrix}-1 & 0.999\\
0.999 & 2
\end{bmatrix},\,G=\begin{bmatrix}0.333\\
0.333
\end{bmatrix},\,H=[1\,\,-1].
\]
This model is stable, although not apriorically guaranteed. Furthermore, this model exhibits a significant decrease in the $H_2$-norm of the error (see the second row of Table \ref{tab_norm}), since matching the derivative of the transfer function is a necessary first-order optimality condition. However, since the interpolation points are fixed, there is no optimality here, whatsoever. On the other hand, using the gradient-based solution for Problem 2 for solving either the KKT system or the partial minimization problem, we compute the optimal  model \eqref{cpend_redmod} in the family \eqref{cpend_redmod}, with $$F=\begin{bmatrix}
-0.2505  &  1.0000 \\  -0.1500  &    0\end{bmatrix},\ G=\begin{bmatrix}
0.2505 \\ 0.1500 \end{bmatrix},\ H=[1\,\,-1],$$ that matches the first two moments at zero and yields a minimal $H_2$-norm of the approximation error in \eqref{cpend_redmod}, see the third row in Table \ref{tab_norm}. Although it is difficult to prove (due to non-convexity of the optimization problem), in principle the ``optimal'' reduced model  is the unique model that minimizes the approximation error in the family \eqref{cpend_redmod}. Finally, employing the gradient-based soution for Problem \ref{prob_optH2_general}   for solving either the KKT system or the partial minimization problem corresponding to it, yields the optimal second order moment matching-based \emph{stable} reduced order model. The matrix $S$ is not fixed, but the algorithm has been initialized with a diagonal matrix having  subunitary  positive  eigenvalues generated at random. The optimal model is obtained interpolating at $0.0109\pm 0.0946j$, which are the eigenvalues of the optimal  matrix
$$S=\begin{bmatrix}
0.0113  &  0.9953\\-0.0090  &  0.0105
\end{bmatrix}$$ given by the gradient algorithm. The optimal approximation is: 
\[
F=\begin{bmatrix}-0.3354  &  0.6486\\-0.3250 &  -0.3060\end{bmatrix},\quad G=\begin{bmatrix}
0.3467 \\ 0.3160
\end{bmatrix}, \quad H=\begin{bmatrix}
1.0049  & -1.0832
\end{bmatrix}.
\]
The reduced order model is stable with $\sigma(F)=-0.1624\pm0.5422$. Note that the matrix $S$ is unstable and since $A$ is stable, $\sigma(S)\cap\sigma(A)=\emptyset$ is satisfied. Furthermore, the resulting $G$ is such that $\sigma(S)\cap\sigma(S-GL)=\emptyset$ is also verified.

\begin{table}[!htbp]
\centering
\label{tab_norm}
\begin{tabular}{|>{\centering}m{9.3cm}|>{\centering}m{4.9cm}|}
\hline
Second order model $\Sigma_{G}$, $G=[g_{1}\,g_{2}]^{T}$ & $H_{2}$-norm of approx. error\tabularnewline
\hline
\hline
$\Sigma_{G}$, $g_{1}=1,\,g_{2}=0.5,$ matching 2 mom. at 0 & $13.91\cdot 10^{-1}$\tabularnewline
\hline
$\Sigma_{G}$, $g_{1}=0.333,\,g_{2}=0.333$ matching 2 mom. at 0 of
$K$ and $K'$ & $1.86\cdot 10^{-1}$\tabularnewline
\hline 
$\Sigma_{G}$, by Problem \ref{prob_optH2_G}, $g_{1}= 0.2507,\,g_{2}=0.15$ matching 2 mom. at 0 & $1.474\cdot 10^{-1}$\tabularnewline
\hline
$\Sigma_{G}$ by Problem \ref{prob_optH2_general}, $g_{1}=0.3467,\,g_{2}=0.3160$ matching 2 mom. at $0.0109\pm 0.0946j$ yielding minimal  $H_2$-norm & $0.6232\cdot 10^{-4}$\tabularnewline
\hline 
\end{tabular}
\caption{$H_{2}$-norms of the approximation errors for different scenarios. }
\end{table}

%%%%%%%%%%%%%%%%%%%%%%%%%%%%%%%%%%%%

\subsection{CD player}
\noindent In the second test we perform model reduction on the CD player system, which  has 120 states, i.e., $n=120$, with a single input and a single output, see, e.g.,\cite{antoulas-2005,gugercin-antoulas-beattie-SIAM2008}. We obtain the optimal model through the solution to Problem \ref{prob_optH2_G} based on the gradient iteration at orders $\nu=1:10$. In Figure \ref{fig_CDplayer_approx_H2_prob2}, we plot the $H_2$-norm of the approximation error versus the reduced order index. We compute the solution to Problem 2 yielding the $H_2$-norm of the approximation errors in families of reduced order models that achieve moment matching at a set of $\nu$ moments at $\nu$ fixed interpolation points. The interpolation points have been chosen at low frequencies and dense, e.g., 0, 0.2, 0.4, 0.6, ..., as well as rare, e.g., 0, 2, 4, 6,... Note that interpolating at zero ensures preservation of DC-gain of the step response of the system. Figure \ref{fig_CDplayer_approx_H2_prob2} also shows that a rare choice of the interpolation points yields better optimal approximations.
\begin{figure}[!htbp]
\centering
\includegraphics[width=.5\textwidth]{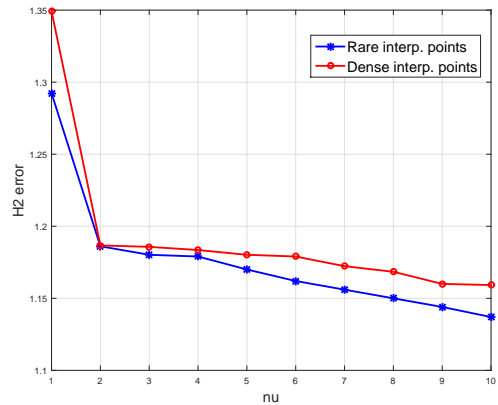}
\caption{$H_2$-norm of the error versus $\nu$ for CD player.}
\label{fig_CDplayer_approx_H2_prob2}
\end{figure}

%%%%%%%%%%%%%%%%%%%%%%%%%%%%%%%

%\subsection{Transportation networks}

%%%%%%%%%%%%%%%%%%%%%%%%%%%%%%%%%5

\section{Conclusions}
\noindent In this paper we have formulated several optimization problems with respect to  $H_2$-norm minimal error approximation in a family of reduced order models that match a prescribed set of fixed $\nu$ moments. For these optimization problems we have derived first-order optimality conditions and  numerical solutions has been proposed  in terms of the gradient method or SDP.  Using test examples from model reduction literature, such as a cart controlled by a double-pendulum or a CD player, % or transportation networks, 
we have also verified the efficiency of our results numerically.


\begin{thebibliography}{10}

\bibitem{anic-beattie-gugercin-antoulas-AUT2013}
{\sc B.~Anic, C.~A. Beattie, S.~Gugercin, and A.~C. Antoulas}, {\em
  Interpolatory weighted-${H}_2$ model reduction}, Automatica, 49 (2013),
  pp.~1275--1280.

\bibitem{antoulas-2005}
{\sc A.~C. Antoulas}, {\em Approximation of large-scale dynamical systems},
  SIAM, Philadelphia, 2005.

\bibitem{antoulas-ball-kang-willems-LINALG1990}
{\sc A.~C. Antoulas, J.~A. Ball, J.~Kang, and J.~C. Willems}, {\em On the
  solution of the minimal rational interpolation problem}, Linear Algebra \&
  its Applications, 137/138 (1990), pp.~511--573.

\bibitem{antoulas-sorensen-1999}
{\sc A.~C. Antoulas and D.~C. Sorensen}, {\em Projection methods for balanced
  model reduction}, tech. report, ECE-CAAM Depts., Rice University, 2001.

\bibitem{astolfi-CDC2007}
{\sc A.~Astolfi}, {\em A new look at model reduction by moment matching for
  linear systems}, in Proceedings of the 46th IEEE Conference on Decision and Control, 2007,
  pp.~4361--4366.

\bibitem{astolfi-TAC2010}
\leavevmode\vrule height 2pt depth -1.6pt width 23pt, {\em Model reduction by
  moment matching for linear and nonlinear systems}, IEEE Transactions on Automatic Control,
  50 (2010), pp.~2321--2336.

\bibitem{astolfi-CDC2010}
\leavevmode\vrule height 2pt depth -1.6pt width 23pt, {\em Model reduction by
  moment matching, steady-state response and projections}, in Proceedings of the 49th IEEE
  Conference on Decision and Control, 2010, pp.~5344 -- 5349.

\bibitem{desouza-bhattacharyya-LAA1981}
{\sc E.~de~Souza and S.~P. Bhattacharyya}, {\em Controllability, observability
  and the solution of ${AX}-{XB}={C}$}, Linear Algebra \& Its Applications, 39 (1981),
  pp.~167--188.

\bibitem{devillemagne-skelton-IJC1987}
{\sc C.~de~Villemagne and R.~E. Skelton}, {\em Model reductions using a
  projection formulation}, Intrtnational Journal of Control, 46 (1987), pp.~2141--2169.

\bibitem{feldman-freund-1995}
{\sc P.~Feldman and R.~W. Freund}, {\em Efficient linear circuit analysis by
  {P}ad\'{e} approximation via a {L}anczos method}, IEEE Transactions on Computer-Aided
  Design, 14 (1995), pp.~639--649.

\bibitem{fujimoto-ono-hayakawa-CDC2010}
{\sc K.~Fujimoto, S.~Ono, and Y~Hayakawa}, {\em Controller reduction for linear
  systems based on subspace balanced truncation}, in Proceedings of the 49th Conference on
  Decision and Control, 2010, pp.~5362--5367.

\bibitem{gallivan-vandooren-NA1999}
{\sc K.~Gallivan and P.~Van Dooren}, {\em Rational approximations of
  pre-filtered transfer functions via the {L}anczos algorithm}, Numerical
  Algorithms, 20 (1999), pp.~331--342.

\bibitem{gallivan-grimme-vandooren-NA1996}
{\sc K.~Gallivan, E.~Grimme, and P.~Van Dooren}, {\em A rational {L}anczos
  algorithm for model reduction}, Numerical Algorithms, 12 (1996), pp.~33--63.

\bibitem{gallivan-vandendorpe-vandooren-SIAM2004}
{\sc K.~Gallivan, A.~Vandendorpe, and P.~Van Dooren}, {\em Model reduction of
  {MIMO} systems via tangential interpolation}, SIAM Journal on Matrix Analysis and Applications, 26
  (2004), pp.~328--349.

\bibitem{gallivan-vandendorpe-vandooren-JCAM2004}
\leavevmode\vrule height 2pt depth -1.6pt width 23pt, {\em Sylvester equations
  and projection based model reduction}, Journal of Computational and Applied Mathematics, 162 (2004),
  pp.~213--229.

\bibitem{glover-1984}
{\sc K.~Glover}, {\em All optimal {H}ankel norm approximations of linear
  multivariable systems and their $l^\infty$-error bounds}, International Journal of Control, 39
  (1984), pp.~1115--1193.

\bibitem{gragg-lindquist-LAA1983}
{\sc W.~B. Gragg and A.~Lindquist}, {\em On the partial realization problem},
  Linear Algebra \& its Applications, 50 (1983), pp.~277--319.

\bibitem{grimme-sorensen-vandooren-NA1995}
{\sc E.~Grimme, D.~Sorensen, and P.~Van Dooren}, {\em Model reduction of state
  space systems via an implicitly restarted {L}anczos method}, Numerical
  algorithms, 12 (1995), pp.~1--31.

\bibitem{grimme-1997}
{\sc E.~J. Grimme}, {\em Krylov projection methods for model reduction}, PhD
  thesis, ECE Dept., Univ. of Illinois, Urbana-Champaign, USA, 1997.

\bibitem{gugercin-antoulas-2004}
{\sc S.~Gugercin and A.~C. Antoulas}, {\em A survey of model reduction by
  balanced truncation and some new results}, International Journal of Control, 77 (2004),
  pp.~748--766.

\bibitem{gugercin-antoulas-beattie-SIAM2008}
{\sc S.~Gugercin, A.~C. Antoulas, and C.~A. Beattie}, {\em ${H}_2$ model
  reduction for large-scale dynamical systems}, SIAM Journal on Matrix Analysis \&
  Applications, 30 (2008), pp.~609--638.

\bibitem{i-TAC2016}
{\sc T.~C. Ionescu}, {\em Two-sided time-domain moment matching for linear
  systems}, IEEE Transactions on Automatic Control, 61 (2016), pp.~2632--2637.

\bibitem{ionescu-astolfi-CDC2013}
{\sc T.~C. Ionescu and A.~Astolfi}, {\em Moment matching based controller
  reduction for linear systems}, in Proceedings of 52nd IEEE Conference on Decision and
  Control, Florence, Italy, 2013, pp.~5528--5533.

\bibitem{i-astolfi-colaneri-SCL2014}
{\sc T.~C. Ionescu, A.~Astolfi, and P.~Colaneri}, {\em Families of moment
  matching based, low order approximations for linear systems}, Systems \&
  Control Letters, 64 (2014), pp.~47--56.

\bibitem{jaimoukha-kasenally-SIAM1997}
{\sc I.~M. Jaimoukha and E.~M. Kasenally}, {\em Implicitly restarted {K}rylov
  subspace methods for stable partial realizations}, SIAM Journal on Matrix Analysis and
  Applications, 18 (1997), pp.~633--652.

\bibitem{KocSti:12}
{\sc M.  Kocvara and M. Stingl}, {\em  PENNON: Software for Linear and Nonlinear Matrix Inequalities},   Handbook on Semidefinite, Conic and Polynomial Optimization, M. Anjos and J. Lasserre (Eds.), Springer (2012), pp. 755--794.

\bibitem{LiLam:11}
{\sc P.~Li, J.~Lam, Z.~Wang, and P.~Date}, {\em Positivity-preserving h model
  reduction for positive systems}, Automatica, 47 (2011), pp.~1504--1511.

\bibitem{Lue:03}
{\sc D.~Luenberger}, {\em Linear and {N}onlinear {P}rogramming}, Kluwer, 2003.

\bibitem{mayo-antoulas-LAA2007}
{\sc A.~J. Mayo and A.~C. Antoulas}, {\em A framework for the solution of the
  generalized realization problem}, Linear Algebra \& Its Applications, 425 (2007),
  pp.~634--662.

\bibitem{meier-luenberger-TAC1967}
{\sc L.~Meier and D.~G. Luenberger}, {\em Approximation of linear constant
  systems}, IEEE Transactions on Automatic Control, 12 (1967), pp.~585--588.

\bibitem{moore-1981}
{\sc B.~C. Moore}, {\em Principal component analysis in linear systems:
  controllability, observability and model reduction}, IEEE Transactions on Automatic
  Control, 26 (1981), pp.~17--32.

\bibitem{Nes:04}
{\sc Yu. Nesterov}, {\em Introductory {L}ectures on {C}onvex {O}ptimization:
  {A} {B}asic {C}ourse}, Kluwer, 2004.

\bibitem{Ran:15}
{\sc A.~Rantzer}, {\em Scalable control of positive systems}, European Journal of
  Control, 24 (2015), pp.~72--80.

\bibitem{ReiVir:09}
{\sc T.~Reis and E.~Virnik}, {\em Positivity preserving balanced truncation for
  descriptor systems}, SIAM Journal on Control and Optimization, 48 (2009),
  pp.~2600--2619.

\bibitem{Toi:85}
{\sc H.~Toivonen}, {\em A globally convergent algorithm for the optimal
  constant output feedback problem}, International Journal of Control, 41 (1985),
  pp.~1589--1599.

\bibitem{vandooren-1995}
{\sc P.~van Dooren}, {\em The {L}anczos algorithm and {P}ad\'{e}
  approximation}.
\newblock Benelux Meeting on Systems and Control, 1995.
\newblock {M}inicourse.

\end{thebibliography}
\end{document}